\RequirePackage{fix-cm}
\documentclass[smallextended]{svjour3}
\usepackage[T1]{fontenc}

\smartqed

\usepackage{algorithm}
\usepackage{algpseudocode}
\usepackage{amsmath,amssymb,amsfonts}
\usepackage{graphicx}
\usepackage{mathtools}
\usepackage{array}
\usepackage{booktabs}
\usepackage{url}
\usepackage{placeins}
\usepackage[hidelinks]{hyperref}

\DeclareMathOperator*{\Argmin}{Argmin}
\DeclareMathOperator*{\Argmax}{Argmax}
\DeclareMathOperator{\Diag}{Diag}
\DeclareMathOperator{\conv}{conv}

\makeatletter
\def\url@leostyle{%
  \@ifundefined{selectfont}{\def\UrlFont{\sf}}{\def\UrlFont{\small\ttfamily}}}
\makeatother
\spnewtheorem{assumption}[theorem]{Assumption}{\bfseries}{\itshape}
\let\springerproof\proof
\let\endspringerproof\endproof
\renewenvironment{proof}{\springerproof}{%
  \ifhmode\qed\else\leavevmode\qed\fi\endspringerproof}

\newcommand{\V}{{\mathrm{Vol}}}

\newcommand{\N}{\mathbb{N}}

\newcommand{\R}{\mathbb{R}}

\newcommand{\K}{\mathcal{K}}
\newcommand{\E}{\mathcal{E}}
\newcommand{\tr}{\mathrm{tr}}

\newcommand{\MaxIE}{\mathrm{MaxIE}}
\newcommand{\Vol}{\mathrm{Vol}}

\newcommand{\wh}{\widehat}
\renewcommand{\S}{\mathbb{S}}

\journalname{Mathematical Programming}

\begin{document}
\title{The Polarity Process for the Maximum-Volume Inscribed Ellipsoid Problem}
\author{Kaizhao Sun}

\institute{Kaizhao Sun 
\at DAMO Academy, Alibaba Group (U.S.) Inc. \\ 
\email{kaizhao.s@alibaba-inc.com}
}

\date{}
\maketitle

\begin{abstract}
We study the maximum-volume inscribed ellipsoid (MaxIE) problem for a polytope
through a geometric iteration based on polarity. Given an interior point, the
method forms the shifted polar polytope, computes its minimum-volume covering
ellipsoid (MinCE), and then polarizes this covering ellipsoid back to obtain a
new inscribed ellipsoid. This procedure was suggested by Khachiyan and Todd in
1993. 
Prior work studied basic polarity identities and gave an asymptotic-convergence
argument for the exact process, but did not furnish a
volume-ratio rate.
Independently of that asymptotic argument, we develop
a new analysis based on the log-volume of the polar MinCE as a
potential function. We prove its convexity with an explicit gradient formula
and establish global linear contraction of the potential gap along each
trajectory. The contraction factor is existential and
instance-dependent. 
This gives both a separate convergence proof and a finite volume-ratio iteration
bound. We further analyze an inexact polarity process in which each MinCE subproblem is solved
only approximately, and derive sufficient oracle tolerances for producing a
prescribed volume approximation. Combining this outer analysis with existing
algorithms for MinCE gives conditional arithmetic estimates based on the path-following Newton method,
the barycentric coordinate descent, and the away-step
Frank--Wolfe method. 
Numerical experiments compare these
oracle choices with two MaxIE baselines and illustrate that performance depends
on matching the MinCE solver to the instance geometry.
\end{abstract}

\addtolength{\abovedisplayskip}{-0.6pt}
\addtolength{\belowdisplayskip}{-0.6pt}
\addtolength{\belowdisplayshortskip}{-0.4pt}
\setlength{\jot}{2.5pt}

\section{Introduction}\label{sec:introduction}
The maximum-volume inscribed ellipsoid (MaxIE) of a polytope is a classical
object in convex geometry and optimization. In this paper we consider a
full-dimensional polytope
\begin{align}\label{eq: polytope}
    P = \{x \in \R^n : a_i^\top x \leq 1,\ \forall i\in[m]\}.
\end{align}
The MaxIE problem asks for the ellipsoid of largest volume contained in $P$,
denoted by $\mathrm{MaxIE}(P)$. For $\gamma\in(0,1]$, we call an inscribed ellipsoid $\E\subseteq P$
\emph{$\gamma$-maximal} if
\[
    \Vol(\E)\geq \gamma\,\Vol(\mathrm{MaxIE}(P)).
\]
The case $\gamma=1$ corresponds to exact maximality. The finite-iteration
complexity statements below are stated for $\gamma\in(0,1)$.
The MaxIE provides a canonical affine rounding of $P$ and plays an important
role in geometric algorithms and convex optimization
\cite{grotschel1988geometric,john1948extremum,khachiyan1990complexity}. In
high-dimensional applications, however, direct computation of MaxIE can be
expensive because the standard convex formulations involve both the center and
the full shape matrix of the ellipsoid.

Khachiyan and Todd \cite{khachiyan1990complexity} gave a foundational
complexity analysis for approximating the MaxIE of $P$ by an algorithm called \texttt{Ellip}.
In the concluding remarks of their 1993 paper, they proposed a natural geometric iteration: from
an interior point $x_k$, compute the minimum-volume covering ellipsoid (MinCE),
also known as the L\"owner ellipsoid,
of the shifted polar polytope $(P-x_k)^\circ$, polarize it back, and use the
shifted center of the resulting inscribed ellipsoid as the next iterate. They asked
whether this sequence converges to the MaxIE of $P$. We call this procedure the \emph{Polarity Process} and present it in Algorithm
\ref{alg: polarity process}. The method produces a new inscribed ellipsoid of
$P$ at each iteration; we use the center--shape parameterization
\eqref{eq: ellip_param} from Section
\ref{sec: preliminaries}.

\begin{algorithm}[h]
\caption{The Polarity Process}
\label{alg: polarity process}
\begin{algorithmic}[1]
\State \textbf{Initialize:} $x_0 \in \mathrm{int}(P)$; \Comment{take any interior point of $P$}
\For{$k = 0, 1, 2\cdots$}
\State compute $(P-x_k)^\circ$; \Comment{take the polar of $P$ w.r.t. $x_k$}
\State $\mathcal{E}_k^\circ := \E(c(x_k), Q(x_k)) \gets \mathrm{MinCE}((P-x_k)^\circ)$; \Comment{compute MinCE of the polar set}
\State $\mathcal{E}'_{k+1}:=\E(x'_{k+1}, E'_{k+1}) \gets \left(\mathcal{E}_k^\circ\right)^\circ$; \Comment{take the polar of MinCE}
\State $\mathcal{E}_{k+1}:= \E(x_{k+1}, E_{k+1})$ with $(x_{k+1},E_{k+1} ) \gets (x_k + x'_{k+1}, E'_{k+1})$; \Comment{shift back to $x_k$}
\EndFor
\end{algorithmic}
\end{algorithm}

We define the polarity mapping as follows. 
\begin{definition}[Polarity mapping]
For $x\in\mathrm{int}(P)$, let $\K(x)$ be the next center produced by one exact
iteration of Algorithm~\ref{alg: polarity process}, i.e., $x_{k+1}=\K(x_k)$.
\end{definition}

X. Andy Sun previously communicated to the author an unpublished draft containing an asymptotic-convergence argument for the polarity process and identifying the center of $\MaxIE(P)$ as the unique fixed point of $\K$. 
To the best of our knowledge, this unpublished argument was the first affirmative resolution of the convergence question posed by Khachiyan and Todd in 1993.
However, his analysis did not derive a quantitative contraction in the natural log-volume potential or a finite volume-ratio bound in the original geometry.
Motivated by this observation, but without relying on the unpublished argument, we develop a new analysis of the polarity process using the
log-volume of the polar MinCE as a potential. Our global linear contraction theorem for the potential gap along each trajectory
proves convergence and yields a finite and instance-dependent volume-ratio
iteration bound. We extend the analysis to inexact MinCE solves and combine it
with existing MinCE algorithms to obtain conditional arithmetic estimates for solving MaxIE problems.

\subsection{Related work}
Most existing algorithms for MaxIE use interior-point or path-following methods
for determinant optimization. The general self-concordant framework of
Nesterov and Nemirovskii \cite{nesterov1994interior}, specialized to MaxIE as in
Khachiyan and Todd \cite{khachiyan1990complexity}, provides the algorithmic
skeleton for such methods and an arithmetic complexity bound\footnote{Suppose $\gamma > 1/\exp(mR)$ so that the outer log is positive.} of
\[
    \mathcal{O}\left(
        m^{4.5}\log\left(
        \frac{mR}{\log(1/\gamma)}\right)
    \right) 
\]
for computing a $\gamma$-maximal ellipsoid, 
where $R> 1$ is a radius ratio for balls bounding $P$ from outside and inside.
The \texttt{Ellip} algorithm of Khachiyan and Todd \cite{khachiyan1990complexity} reduces MaxIE to a sequence of maximal-paraboloid subproblems and
solves these subproblems by path-following Newton methods.
Their overall arithmetic complexity bound is
\[
    \mathcal{O}\left(
        m^{3.5}\log\left(
        \frac{mR}{\log(1/\gamma)}\right)
        \left[1+\log^+\!\left(
        \frac{n\log R}{\log(1/\gamma)}\right)\right]
    \right).
\]
The second logarithmic term arises from solving a sequence of subproblems.
Subsequent work improved the worst-case complexity of direct MaxIE computation:
Anstreicher \cite{anstreicher2002improved} avoided the auxiliary subproblem
sequence in the Khachiyan--Todd formulation and obtained the bound
\[
    \mathcal{O}\left(
        m^{3.5}\log\left(
        \frac{mR}{\log(1/\gamma)}\right)
    \right).
\]
Zhang and Gao \cite{zhanggao2003mvie} developed practical primal-dual
interior-point algorithms, including dense formulations with leading
$\mathcal{O}(m^3)$ Newton-step cost per iteration, and compared their formulations
with a modified Khachiyan--Todd approach; however, they did not give a comparable
overall $\gamma$-dependent worst-case arithmetic complexity bound. Works
\cite{anstreicher2002improved,zhanggao2003mvie}
should be viewed as direct MaxIE
interior-point methods, rather than as polarity iterations or reductions to
MinCE oracles. They give polynomial-time algorithms or efficient practical
methods, but their Newton systems can be expensive in the large-scale regimes
that motivate this paper.

The polarity process originates from the geometric connection between inscribed
and covering ellipsoids. A related polarity-based idea was studied by Xie et al.~\cite{xie2006maxsphere} for approximating the maximum
inscribed sphere of a high-dimensional polytope under bounded aspect-ratio
assumptions. Their analysis reduces the sphere problem to a sequence of minimum
enclosing sphere problems and relies on core-set and Euclidean geometric
arguments. The ellipsoidal problem studied in our paper additionally requires
control of both the center and the shape matrix. The closest prior analysis is
the aforementioned unpublished work by X. Andy Sun.
He derived the basic MaxIE--MinCE polarity identities, the unique fixed-point characterization, 
and monotonicity of the returned ellipsoid volumes, together with continuity and asymptotic-convergence arguments.
He also invoked a theorem by Bessaga \cite{bessaga1959converse}, which gives contraction in an unspecified complete metric but does not yield a log-volume contraction, volume-ratio bound, or arithmetic complexity estimate. Our analysis is independent of that asymptotic argument and uses a log-volume potential, differential
properties of the polar MinCE value function, and perturbation bounds for
inexact MinCE solves.

The literature on MinCE supplies the computational mechanisms used in
this paper. Khachiyan \cite{khachiyan1996rounding} studied rounding and the
L\"owner ellipsoid problem in the real-number model of computation. His
lifting technique is the transformation used here for the polar
MinCE subproblems. Khachiyan introduced a barycentric coordinate descent (BCD) method and gave
explicit arithmetic complexity bounds; he also gave a path-following route for
approximating the L\"owner ellipsoid. Sun and Freund
\cite{sunfreund2004mvce} developed practical algorithms for computing
minimum-volume covering ellipsoids, including a dual reduced Newton method and
active-set strategies that perform well on large data sets. Ahipasaoglu et al. \cite{ahipasaoglu2008fw} later interpreted the corresponding dual updates as a Frank--Wolfe method and analyzed a modified version with
away steps, the Wolfe--Atwood--Todd--Yildirim (WA-TY) method, proving
fixed-instance linear convergence for the dual MinCE problem.

Recent work on centered John ellipsoids uses leverage-score iterations,
equivalently D-optimal design. For $m$ constraints in dimension $d$, the
certified uniform average in \cite{CohenCousinsLeeYang2019} requires
$\mathcal{O}(\epsilon^{-1}\log(m/d))$ rounds. Lazy-update, streaming, sketching,
sparsity, and small-treewidth variants reduce per-round cost but retain this
$1/\epsilon$ dependence \cite{WoodruffYasuda2024,CaoLiSongYangZhou2025}. Zhao \cite{Zhao2025AwayStep}
analyzed a distinct method in the same broad away-step Frank--Wolfe
family and proved global linear convergence in objective and Frank--Wolfe gap. Li et al. \cite{LiYuJiangGaoHan2026} traced the $1/\epsilon$ factor to uniform
averaging: with exact scores and a nondegenerate optimum, their accelerated and
facial-Newton phases require $C+\mathcal{O}(\sqrt{\kappa}\log(1/\epsilon))$
and $C+\mathcal{O}(d^2\log\log(1/\epsilon))$ queries, respectively, where 
$\kappa$ is the facial Hessian condition number and $C$ is an instance-dependent setup cost.
These fixed-instance advances complement our moving-center analysis.


\subsection{Our contributions}
We develop the following new algorithmic theory of the polarity process.

\begin{itemize}
    \item \textbf{New potential analysis.} We analyze the exact process through
    \[
        v(x)=\log \Vol\bigl(\mathrm{MinCE}((P-x)^\circ)\bigr),
    \]
    proving convexity, differentiability, and the gradient identity
    \[
        \nabla v(x)=(n+1)c(x),
    \]
    where $c(x)$ is the polar MinCE center. Local quadratic growth and a
    Polyak–Łojasiewicz-type inequality yield instance-dependent global linear contraction of the potential gap along each trajectory and
    an iteration bound for producing a $\gamma$-maximal inscribed ellipsoid.

    \item \textbf{Inexact polarity process.} We allow additive log-determinant
    error in each polar MinCE subproblem and give sufficient oracle tolerances
    based on the target volume ratio $\gamma$ and initial-sublevel geometry,
    preserving the outer guarantee.

    \item \textbf{Conditional arithmetic estimates with concrete MinCE oracles.} We combine the outer
    inexact analysis with existing MinCE algorithms. Khachiyan's results
    \cite{khachiyan1996rounding} give bounds for BCD
    and path-following Newton oracles, while the WA-TY fixed-instance theory
    \cite{ahipasaoglu2008fw} supplies the iteration bound used in our
    realized-trajectory work accounting.
    With the path-following Newton method as the MinCE
    oracle, the polarity process has the conditional arithmetic estimate
    \[
        \mathcal{O}\left(
            m^{3.5}\log\left(\frac{m^3}{\log(1/\gamma)^2}\right)
            \left[1+\log^+\left(\frac{2n\log R}{\log(1/\gamma)}\right)\right]
        \right).
    \]
    In the high-accuracy regime,
    $\log(m^3/\log(1/\gamma)^2)=\mathcal O(\log(m/\log(1/\gamma)))$; unlike
    \texttt{Ellip}, the displayed factor has no explicit $R$-dependence when the
    outer geometric constants are fixed.
    This does not improve Anstreicher's direct worst-case bound; it shows
    that the above estimate has the same
    $m^{3.5}$ factor with a sequence of MinCE
    subproblems.
    With the BCD oracle,
    the polarity process has the arithmetic estimate
    \[
        \mathcal{O}\left(
            \left(\frac{m^3n^3}{\log(1/\gamma)^2}+mn^2\log\log m\right)
            \left[1+\log^+\left(\frac{2n\log R}{\log(1/\gamma)}\right)\right]
        \right).
    \]
    The BCD oracle is a clean theoretical baseline but has high target-accuracy
    dependence. The away-step Frank--Wolfe oracle from \cite{ahipasaoglu2008fw}
    exploits sparse-support updates and away steps and gives the
    realized-trajectory estimate
    \[
        \mathcal{O}\left(
            m(n+1)^2
            \left(
                J_\gamma+
                C_{\mathrm{WA},\gamma}\log\left(\frac{m^2(n+1)}{\log(1/\gamma)^2}\right)
            \right)
            \left[1+\log^+\left(\frac{2n\log R}{\log(1/\gamma)}\right)\right]
        \right),
    \]
    where $J_\gamma$ and $C_{\mathrm{WA},\gamma}$ are instance-dependent
    constants for the realized WA-TY trajectory. This a posteriori accounting
    has the smallest displayed explicit polynomial prefactor in $m$ and $n$
    among these estimates. All derived bounds hide outer
    constants set by the initial sublevel geometry.

    \item \textbf{Numerical case studies.} Four polytope families illustrate that
    the process could be competitive with geometry-appropriate MinCE solvers.
\end{itemize}

\subsection{Notation and organization}\label{sec: notation}
For $p\in\N$, let $[p]=\{1,\ldots,p\}$. Write $\R^n$ for Euclidean space,
$\langle x,y\rangle=x^\top y$, and $\|x\|$ for its norm. 
Let $\bar B_2(x,r)$ and $B_2(x,r)$ denote the closed and open Euclidean balls
centered at $x$ with radius $r>0$. Use $\mathrm{int}$, $\conv$,
and $(\cdot)^\circ$ for interior, convex hull, and
origin polarity after translation, respectively. 

Matrix operator and Frobenius norms are $\|A\|,\|A\|_F$; $\mathbf0$ and $I_n$ denote zero and
identity matrices or vectors of clear dimension. Use $\S^n_{++}(\S^n_+)$ for the positive-(semi)definite cone,
and $\tr,\det$ are trace and determinant. For $t>0$, write $\log^+(t):=\max\{0,\log t\}$ and set $\log^+(0):=0$.
Other symbols are defined when first used.

Section~\ref{sec: preliminaries} provides background material and prior
results. Sections \ref{sec:convergence-polarity-process} and
\ref{sec:inexact-polarity-process} analyze exact and inexact processes with complexity estimates, respectively, 
followed by experiments in Section \ref{sec:numerical-experiments}. Section \ref{sec:conclusion} concludes this paper.

\section{Preliminaries}\label{sec: preliminaries}
\subsection{Parameterization of an ellipsoid}
An ellipsoid $\E\subset\R^n$ is an affine image of the unit ball:
\begin{subequations}\label{eq: ellip_param}
\begin{align}
    \E &= \{Bu + c : \|u\|_2 \leq 1\} \label{eq: ellip_param_1}\\
       &= \{x : (x-c)^\top B^{-2}(x-c) \leq 1\}, \label{eq: ellip_param_2}
\end{align}
\end{subequations}
where $B\in\S^n_{++}$ and $c\in\R^n$. The volume of $\E$ is $(\det B)\mu_n$, where $\mu_n$
is the volume of the unit Euclidean ball in $\R^n$. We use the normalized volume:
\begin{align}\label{eq: scaled_ellip_volume}
    \mathrm{Vol}(\E) = \det B.
\end{align}
We denote the ellipsoid in the form of \eqref{eq: ellip_param_2} by $\E(c,B^{-2})$; more generally,
for some inverse-shape matrix $Q\in \S^n_{++}$, we write
\begin{align}
    & \E(c,Q)=  \{y\in\R^n:(y-c)^\top Q(y-c)\le1\}, \\
    & \Vol(\E(c,Q))= \det(Q)^{-1/2}.
\end{align}

\subsection{The MinCE problem, convex reformulation, and dual problem}
\paragraph{The MinCE Problem.} Fix $x \in \operatorname{int}(P)$. Define
\begin{align}\label{eq: polar_extreme_points}
    \tilde{a}_i(x) = \frac{a_i}{1-a_i^\top x}, \qquad i\in[m].
\end{align}
Then $(P-x)^\circ=\conv\{\tilde a_i(x)\}_{i=1}^m$. For an irredundant
representation of $P$, these are precisely extreme points of $(P-x)^\circ$.

A standard formulation of $\mathrm{MinCE}((P-x)^\circ)$ is
\begin{subequations}\label{eq: MinCE}
\begin{align}
    (c(x), Q(x)) \in \Argmin_{c\in \R^n, Q\in \mathbb{S}^n_{++}}\quad & -\log \det Q \\
    \mathrm{s.t.}\quad & (\tilde{a}_i(x) - c)^\top Q(\tilde{a}_i(x) - c) \leq 1, ~ \forall i\in[m].
\end{align}
\end{subequations}
The MinCE optimizer $\E(c(x),Q(x))$ has normalized log-volume
$-\frac12\log\det Q$.

\paragraph{Convex Reformulation in $\R^{n+1}$.} Although \eqref{eq: MinCE} is
not jointly convex in $(c,Q)$, lifting yields the following convex reformulation:
\begin{subequations}\label{eq: lifted_MinCE}
\begin{align}
    \min_{\mathcal{Q} \in \S^{n+1}_{++}} \quad & -\log \det \mathcal{Q} \\
    \mathrm{s.t.}\quad &
    \begin{bmatrix}
        \tilde{a}_i(x)\\ 1
    \end{bmatrix}^\top
    \mathcal{Q}
     \begin{bmatrix}
        \tilde{a}_i(x)\\ 1
    \end{bmatrix} \leq 1, ~\forall i\in[m].
\end{align}
\end{subequations}
This standard lifting is used in \cite{khachiyan1996rounding}; see also \cite{khachiyan1990complexity,nesterov1994interior,sunfreund2004mvce}. Write the optimizer as
\begin{align} \label{eq: lifted_MinCE_sol}
    \mathcal{Q}(x) = \begin{bmatrix}
            G(x) & z(x) \\ z(x)^\top & s(x)
    \end{bmatrix}
    \in \S^{n+1}_{++}.
\end{align}
Its height-one section is the MinCE, recovered by
\begin{align}\label{eq: MinCE_primal_sol_from_reformulation}
    c(x) = -G(x)^{-1}z(x), \quad
    Q(x) = \frac{G(x)}{1-s(x)+ z(x)^\top G(x)^{-1} z(x)}.
\end{align}
\paragraph{Dual MinCE Problem.} The dual MinCE problem can be written as
\begin{align}\label{eq: dual_MinCE}
    \max_{\lambda\in \Delta_m} ~ \log \det
    \begin{bmatrix}
        \sum_{i=1}^m \lambda_i \tilde{a}_i(x)\tilde{a}_i(x)^\top & \sum_{i=1}^m \lambda_i \tilde{a}_i(x) \\
        \sum_{i=1}^m \lambda_i \tilde{a}_i(x)^\top & 1
    \end{bmatrix},
\end{align}
where $\Delta_m:=\{\lambda\in\R^m_+:\sum_i\lambda_i=1\}$. From any dual
optimizer $\lambda^*(x)$, the primal optimal solution can be recovered via
\begin{subequations}\label{eq: MinCE_primal_sol}
\begin{align}
    c(x) = & \sum_{i=1}^m \lambda_i^*(x)\tilde a_i(x),\\
    Q(x) =&  \frac{1}{n}\left(\sum_{i=1}^m \lambda_i^*(x)\tilde a_i(x)\tilde a_i(x)^\top
    - c(x)c(x)^\top\right)^{-1}.
\end{align}
\end{subequations}

We conclude this section with the following useful fact; see, e.g.,
\cite{aubrun2017alice,kurzhanski1997ellipsoidal}.
\begin{lemma}\label{lemma: useful_lemma}
Let $\E=\E(c,Q)$ contain the origin in its interior. Its polar with respect to the origin is
\[
    \E^\circ
    =
    \E\!\left(-\frac{Qc}{1-c^\top Q c},
    (1-c^\top Q c)(Q^{-1}-cc^\top)\right).
\]
Moreover,
\[
    \V(\E)=\frac{1}{(\det Q)^{1/2}},
    \qquad
    \V(\E^\circ)=\frac{(\det Q)^{1/2}}{(1-c^\top Q c)^{(n+1)/2}},
\]
and hence
\begin{align*}
    \V(\E)\V(\E^\circ)
    &=\frac{1}{(1-c^\top Q c)^{(n+1)/2}},\\
    \log\V(\E)+\log\V(\E^\circ)
    &=\frac{n+1}{2}\log\!\left(\frac{1}{1-c^\top Q c}\right).
\end{align*}
The two polar ellipsoids are centered at the origin if and only if $c^\top Q c=0$;
equivalently, their volume product is $1$, or their log-volume sum is $0$.
\end{lemma}

\section{Convergence of the polarity process}\label{sec:convergence-polarity-process}
Our new analysis uses the potential
\begin{align}\label{eq: potential function}
    v(x) := \log \V (\mathrm{MinCE}((P-x)^\circ)).
\end{align}
We first introduce the notation for ease of later presentation. 

\subsection{Notation and abbreviations}\label{sect: proof_notations}
Fix $x\in\mathrm{int}(P)$ and suppress the iteration subscript. With context
clear, abbreviate the quantities in \eqref{eq: polar_extreme_points} and \eqref{eq: MinCE} by
\begin{alignat}{4}\label{eq: MinCE_abbrev}
    \tilde{a}_i & := \tilde{a}_i(x), \qquad
    & c &:= c(x), \qquad
    & Q &:= Q(x), \qquad
    & \tau(x) &:= c(x)^\top Q(x)c(x).
\end{alignat}
For the dual MinCE problem \eqref{eq: dual_MinCE}, define
\begin{subequations}\label{eq: dual_MinCE_quantities}
\begin{align}
    A(x, \lambda) :=  & \sum_{i=1}^m \lambda_i \tilde{a}_i(x) \tilde{a}_i(x)^\top, \quad
    c(x, \lambda) :=  \sum_{i=1}^m \lambda_i \tilde{a}_i(x), \\
    S(x, \lambda) :=  & A(x,\lambda) - c(x,\lambda)c(x, \lambda)^\top.
\end{align}
\end{subequations}
Define the full-support-span domain
\begin{align}
    \Delta_m^{\mathrm{fs}}
    :=
    \left\{\lambda\in\Delta_m:
    \operatorname{span}\left\{
    \begin{bmatrix}a_i\\1\end{bmatrix}:\lambda_i>0
    \right\}=\R^{n+1}\right\}.
\end{align}
The same spanning condition holds with
$[a_i^\top,1]^\top$ replaced by $[\tilde a_i(x)^\top,1]^\top$, because the
two families differ by nonzero projective scalings and an invertible linear
map. Thus $S(x,\lambda)\succ0$ exactly when
$\lambda\in\Delta_m^{\mathrm{fs}}$.
Let
\begin{subequations}\label{eq: dual_MinCE_solution_set}
\begin{align}
    \lambda^* \in & \Lambda(x), \\
    ~\text{where}~ \Lambda(x) := &  \Argmax_{\lambda \in \Delta_m^{\mathrm{fs}}} \log \det
    \begin{bmatrix}
        A(x, \lambda) & c(x, \lambda) \\
        c(x, \lambda)^\top & 1
    \end{bmatrix},
\end{align}
\end{subequations}
be dual-optimal, and define
\begin{align}\label{eq: AcS_with_optimal_lambda}
        A(x) := A(x, \lambda^*), \quad
        c(x) := c(x, \lambda^*), \quad
        S(x) := S(x, \lambda^*).
\end{align}
Although the dual optimizer need not be unique,
\eqref{eq: MinCE_primal_sol} gives
$c(x,\lambda^*)=c(x)$ and $S(x,\lambda^*)=Q(x)^{-1}/n$, where the right-hand
side is determined by the unique primal MinCE. Since
$A=S+cc^\top$, the three quantities in \eqref{eq: AcS_with_optimal_lambda}
are independent of the selected dual optimizer.

For $h\in\R^n$, choose $\bar t>0$ with
$x+th\in\mathrm{int}(P)$ for $t\in(-\bar t,\bar t)$, and set
\begin{align}\label{eq: x_along_h}
    x(t) := x + th,~ t\in(-\bar t,\bar t),
\end{align}
and, holding $\lambda^*\in\Lambda(x)$ fixed, parameterize
\begin{alignat}{2}\label{eq: parameterization_in_t}
     \tilde{a}_i(t) & :=  \tilde{a}_i(x(t)),\quad 
    A(t) && :=  A(x(t), \lambda^*), \notag \\ 
     c(t) &:=  c(x(t), \lambda^*),~
    S(t) && :=  S(x(t), \lambda^*).
\end{alignat}
We suppress $t$ at $t=0$; overdots denote the corresponding derivatives.

\subsection[Convexity and gradient of v]{Convexity and gradient of $v$}
We prove convexity and differentiability of $v$ through its dual MinCE value.
\begin{lemma}\label{lemma: v_in_S}
The function $v(x)$ can be expressed as
\begin{align*}
    v(x)  =  & \frac{1}{2} \log \det S(x) + \frac{n}{2}\log n
    = \max_{\lambda \in \Delta_m^{\mathrm{fs}}}
    \frac{1}{2} \log \det S(x, \lambda)  +  \frac{n}{2}\log n.
\end{align*}
\end{lemma}
\begin{proof}
    By the Schur complement, we have
    \begin{align*}
         \det
    \begin{bmatrix}
        A(x, \lambda) & c(x, \lambda) \\
        c(x, \lambda)^\top & 1
    \end{bmatrix}
    = \det (A(x,\lambda) - c(x,\lambda)c(x, \lambda)^\top) = \det S(x, \lambda).
    \end{align*}
    The dual MinCE problem \eqref{eq: dual_MinCE} is therefore equivalent to
    \begin{align*}
        \max_{\lambda \in \Delta_m^{\mathrm{fs}}} \log \det S(x, \lambda).
    \end{align*}
    For any optimizer $\lambda^*$, $S(x)=S(x,\lambda^*)$, so
    \begin{align*}
        \log \det S(x) = \max_{\lambda \in \Delta_m^{\mathrm{fs}}} \log \det S(x, \lambda).
    \end{align*}
    Moreover,
    \begin{align}
        v(x) = & \log \V(\mathrm{MinCE}((P-x)^\circ)) =  \log \det ( Q(x)^{-1/2}) =  -\frac{1}{2}\log \det Q(x) \notag \\
        = & \frac{1}{2} \log \det S(x) + \frac{n}{2}\log n, \notag
    \end{align}
    where the last equality uses $Q(x)=\frac1nS(x)^{-1}$ from
    \eqref{eq: MinCE_primal_sol} and
    \begin{align*}
        \det Q(x) = \det \frac{1}{n}S(x)^{-1} = (n^n \det S(x))^{-1}.
    \end{align*}
    This completes the proof.
\end{proof}

\begin{lemma}\label{lemma: convexity_of_F}
    Fix $\lambda \in \Delta_m^{\mathrm{fs}}$. The following function is convex in $x$:
    \begin{align*}
        F(x, \lambda):= \frac{1}{2}\log \det S(x, \lambda)
    \end{align*}
\end{lemma}
\begin{proof}
    Define lifted vectors:
    \begin{align*}
        y_i(x) = \begin{bmatrix}
            \tilde{a}_i(x)\\ 1
        \end{bmatrix} \in \R^{n+1},~
        Y(x):= [y_1(x),\cdots,y_m(x)] \in \R^{(n+1)\times m}.
    \end{align*}
    Let $\Lambda = \mathrm{Diag}(\lambda)$. Then it holds that
    \begin{align*}
        Y(x) \Lambda Y(x)^\top = \sum_{i=1}^m \lambda_i y_i(x)y_i(x)^\top =
        \begin{bmatrix}
            A(x, \lambda) & c(x, \lambda) \\
            c(x, \lambda)^\top & 1
        \end{bmatrix}.
    \end{align*}
    By the Schur complement,
    \begin{align}
        \det (Y(x) \Lambda Y(x)^\top) =
        \det(A(x, \lambda) - c(x, \lambda)c(x, \lambda)^\top) = \det S(x, \lambda). \notag
    \end{align}
    Let $B:=Y(x)\Lambda^{1/2}\in\R^{(n+1)\times m}$. Cauchy--Binet
    \cite[Section 0.8.7]{horn2012matrix} gives
    \begin{align}\label{eq: cauchy_binet_1}
         \det S(x,\lambda) = & \det (BB^\top) = \sum_{I\subseteq [m], |I|=n+1} \det(B_I) \det(B_I^\top) \notag \\
         = & \sum_{I\subseteq [m], |I|=n+1} \det(B_I)^2,
    \end{align}
    where $B_I$ is the square submatrix formed by columns indexed by $I$.
    Full-dimensionality gives $m\ge n+1$, and diagonal $\Lambda^{1/2}$ gives
    \begin{align}\label{eq: cauchy_binet_2}
        \det(B_I) = \det(Y_I) \prod_{i\in I}\sqrt{\lambda_i},
    \end{align}
    where $Y_I$ is formed analogously. For $i\in I$,
    \begin{align*}
        y_i(x) = & \begin{bmatrix}
            \tilde{a}_i(x)\\ 1
        \end{bmatrix} =  \frac{1}{1-a_i^\top x}
        \begin{bmatrix}
            a_i \\ 1-a_i^\top x
        \end{bmatrix}
        = \frac{1}{1-a_i^\top x}
        \begin{bmatrix}
            I_n & \mathbf{0} \\
            -x^\top & 1
        \end{bmatrix}
        \begin{bmatrix}
            a_i \\ 1
        \end{bmatrix}.
    \end{align*}
    For $I:=\{i_1,\ldots,i_{n+1}\}$, write $Y_I$ as
    \begin{align*}
        Y_I =
        \begin{bmatrix}
            I_n & \mathbf{0} \\
            -x^\top & 1
        \end{bmatrix}
        \underbrace{
        \begin{bmatrix}
            a_{i_1} & \cdots & a_{i_{n+1}} \\
            1 & \cdots & 1
        \end{bmatrix}
        }_{H_I}
        \begin{bmatrix}
            \frac{1}{1-a_{i_1}^\top x} & & \\
            & \ddots& \\
            & & \frac{1}{1 - a_{i_{n+1}}^\top x}
        \end{bmatrix}.
    \end{align*}
    The first factor is triangular with determinant $1$; hence
    \begin{align}\label{eq: cauchy_binet_3}
        \det(Y_I) = \det(H_I) \prod_{i\in I} \frac{1}{1-a_i^\top x}.
    \end{align}
    Combining \eqref{eq: cauchy_binet_1}--\eqref{eq: cauchy_binet_3} gives
    \begin{align*}
        & \det S(x, \lambda) = \sum_{\substack{I\subseteq[m]\\|I|=n+1}} \beta_I(\lambda) \prod_{i\in I} \frac{1}{(1-a_i^\top x)^2}, \\
        & \text{where}~ \beta_I(\lambda):=  \left(\prod_{i\in I}\lambda_i\right)\det(H_I)^2 \geq 0.
    \end{align*}
    Membership in $\Delta_m^{\mathrm{fs}}$ ensures that at least one
    $\beta_I(\lambda)$ is positive. Hence
    \begin{align*}
        F(x, \lambda) = \frac{1}{2}\log \left( \sum_{\stackrel{I\subseteq[m], |I|=n+1, }{\beta_I(\lambda) > 0}} \exp\left( \sum_{i\in I} - 2\log(1-a_i^\top x) + \log \beta_I(\lambda) \right)\right).
    \end{align*}
    Convexity of $-\log t$ on $t>0$, affinity of $1-a_i^\top x$, and
    log-sum-exp convexity show that $F(\cdot,\lambda)$ is convex.
\end{proof}

The required first-order formulas are as follows.
\begin{lemma}\label{lemma: differentiation}
    For $h\in\R^n$, let $x(t)$ and the associated quantities be as in
    \eqref{eq: x_along_h} and \eqref{eq: parameterization_in_t}. At $t=0$, it holds that 
    \begin{align*}
        \dot{\tilde{a}}_i = & (\tilde{a}_i^\top h)\tilde{a}_i, \quad \dot{c} = Ah, \\
        \dot{A} = & 2 \sum_{i=1}^m \lambda_i^* (\tilde{a}_i^\top h) \tilde{a}_i \tilde{a}_i^\top, \\
        \dot{S} = & 2 \sum_{i=1}^m \lambda_i^* (\tilde{a}_i^\top h) \tilde{a}_i \tilde{a}_i^\top - Ahc^\top -c h^\top A.
    \end{align*}
\end{lemma}
\begin{proof}
    \begin{itemize}
    \item \textbf{Differentiating} $\tilde{a}_i$.
    From $\tilde{a}_i(t) = \frac{a_i}{1 - a_i^\top (x+ th)}$, it holds that 
    \begin{align}\label{eq: differentiating_tilde_a}
        \dot{\tilde{a}}_i =\frac{a_i^\top h}{(1-a_i^\top x - ta_i^\top h)^2} a_i\Bigg|_{t=0} = \frac{a_i^\top h}{(1-a_i^\top x)^2}a_i = (\tilde{a}_i^\top h)\tilde{a}_i.
    \end{align}
    \item \textbf{Differentiating} $c$. Since
    $c(t)=\sum_i\lambda_i^*\tilde a_i(t)$, it holds that 
    \begin{align}\label{eq: differentiating_c}
        \dot{c} = \sum_{i=1}^m \lambda^*_i \dot{\tilde{a}}_i
        \stackrel{\eqref{eq: differentiating_tilde_a}}{=} & \sum_{i=1}^m \lambda_i^* (\tilde{a}_i^\top h)\tilde{a}_i  = \sum_{i=1}^m \lambda_i^* (\tilde{a}_i\tilde{a}_i^\top) h= Ah.
    \end{align}
    \item \textbf{Differentiating} $A$. The product rule applied to $A(t)$ gives
    \begin{align}\label{eq: differentiating_A}
        \dot{A} = \sum_{i=1}^m \lambda_i^*(\dot{\tilde{a}}_i \tilde{a}_i^\top + \tilde{a}_i\dot{\tilde{a}}_i^\top) \stackrel{\eqref{eq: differentiating_tilde_a}}{=} 2 \sum_{i=1}^m \lambda_i^* (\tilde{a}_i^\top h) \tilde{a}_i \tilde{a}_i^\top.
    \end{align}
    \item \textbf{Differentiating} $S$. Applying the product rule to $S(t)$ gives
    \begin{align}
        \dot{S} = \dot{A} - \dot{c}c^\top - c \dot{c}^\top
            \stackrel{\eqref{eq: differentiating_c}\eqref{eq: differentiating_A}}{=} 2 \sum_{i=1}^m \lambda_i^* (\tilde{a}_i^\top h) \tilde{a}_i \tilde{a}_i^\top - Ahc^\top -c h^\top A.  \notag
    \end{align}
    \end{itemize}
    This completes the proof. 
\end{proof}
\begin{proposition}\label{proposition: v_gradient}
    The function $v$ is convex and differentiable on $\mathrm{int}(P)$, with gradient given by $\nabla v(x) = (n+1)c(x)$.
\end{proposition}
\begin{proof}
    Lemmas \ref{lemma: v_in_S} and \ref{lemma: convexity_of_F} make $v$ a
    pointwise maximum of convex functions, and hence convex.  Define 
    \[
        D(x,\lambda):=\det S(x,\lambda),
        \qquad
        d(x):=\max_{\lambda\in\Delta_m}D(x,\lambda)>0.
    \]
    The maximizers of $D(x,\cdot)$ are precisely $\Lambda(x)$ and belong to
    $\Delta_m^{\mathrm{fs}}$. Since $D$ is smooth and $\Delta_m$ is compact,
    Danskin's theorem \cite[Proposition~B.22]{bertsekas1999nonlinear} and the
    chain rule applied to $v(x)=\frac12\log d(x)+\frac n2\log n$ give
    \begin{align}
        v'(x;h)
        &=\frac{1}{2d(x)}
          \max_{\lambda^*\in\Lambda(x)}D_xD(x,\lambda^*)[h]
          =\max_{\lambda^*\in\Lambda(x)}D_xF(x,\lambda^*)[h].\notag
    \end{align}
    We now evaluate the last derivative. First,
    \begin{align*}
        D_xF(x, \lambda^*)[h]
        &= \frac{d}{dt}F(x+th,\lambda^*)\Big|_{t=0} \\
        &= \frac{d}{dt}\left(\frac12\log\det S(t)\right)\Big|_{t=0}
        = \frac12\tr(S^{-1}\dot S).
    \end{align*}
    Lemma \ref{lemma: differentiation} and the fact that $\tr(xy^\top)=y^\top x$
    give
    \begin{align}\label{eq: DxF_as_trace_2}
        D_xF(x, \lambda^*)[h] = & \frac{1}{2}\tr\left(S^{-1} \left(2 \sum_{i=1}^m \lambda_i^* (\tilde{a}_i^\top h) \tilde{a}_i \tilde{a}_i^\top - Ahc^\top -c h^\top A \right)\right) \notag\\
        = & \sum_{i=1}^m \lambda_i^* (\tilde{a}_i^\top h) \tilde{a}_i ^\top S^{-1} \tilde{a}_i - c^\top S^{-1} Ah.
    \end{align}
    For $i\in[m]$ with $\lambda_i^*>0$,
    dual optimality, complementary slackness, and $Q=\frac1nS^{-1}$ give
    \begin{align*}
        1 = (\tilde{a}_i - c)^\top Q (\tilde{a}_i - c) = \frac{1}{n}(\tilde{a}_i - c)^\top S^{-1} (\tilde{a}_i - c).
    \end{align*}
    Expanding gives
    \begin{align}\label{eq: a_S_a_expanded}
        \tilde{a}_i^\top S^{-1}\tilde{a}_i = n + 2c^\top S^{-1}\tilde{a}_i -c^\top S^{-1}c.
    \end{align}
    Substituting \eqref{eq: a_S_a_expanded} into
    \eqref{eq: DxF_as_trace_2}, with zero-weight terms vanishing, yields
    \begin{align}
        & \sum_{i=1}^m \lambda_i^* (\tilde{a}_i^\top h) \tilde{a}_i ^\top S^{-1} \tilde{a}_i  \notag \\
        = & (n-c^\top S^{-1}c )\sum_{i=1}^m \lambda_i^* (\tilde{a}_i^\top h) + 2 \sum_{i=1}^m \lambda_i^* (\tilde{a}_i^\top h) (c^\top S^{-1}\tilde{a}_i) \notag\\
        = & (n-c^\top S^{-1}c) \left(\sum_{i=1}^m \lambda^*_i \tilde{a}_i\right)^\top h + 2c^\top S^{-1} \left(\sum_{i=1}^m \lambda_i^*\tilde{a}_i\tilde{a}_i^\top\right)h \notag \\
        = & (n-c^\top S^{-1}c)c^\top h + 2c^\top S^{-1}Ah,\label{eq: first_term_in_derivative_of_F}
    \end{align}
    Combining \eqref{eq: DxF_as_trace_2} and
    \eqref{eq: first_term_in_derivative_of_F}, with $S=A-cc^\top$, gives
    \begin{align*}
        D_xF(x, \lambda^*)[h] = & (n-c^\top S^{-1}c)c^\top h + c^\top S^{-1}Ah \\
        = & n c^\top h + c^\top S^{-1}(A - cc^\top)h\\
        = & n c^\top h + c^\top S^{-1}S h = (n+1) c^\top h.
    \end{align*}
    The unique MinCE center gives
    $v'(x;h)=(n+1)c(x)^\top h$, and therefore we conclude that $\nabla v(x)$ is exactly $(n+1)c(x)$.
\end{proof}

\subsection[Local quadratic growth of v around the MaxIE center]{Local quadratic growth of $v$ around the MaxIE center}
We prove that $v$ grows quadratically near the MaxIE center $x^*$: some
neighborhood $U_{\mathrm{QG}}$ and $\mu>0$ satisfy
\begin{align*}
    v(x) - v(x^*) \geq \mu \|x-x^*\|^2
\end{align*}
for all $x\in U_{\mathrm{QG}}$. In this subsection, fix some $\lambda^*\in\Lambda(x^*)$ and define
\begin{align}\label{eq: valued_function_with_optimal_dual}
    g(x) := F(x, \lambda^*) + \frac{n}{2}\log n.
\end{align}
\begin{lemma}\label{lemma: facts_about_g}
Let $x^*$ be the center of $\MaxIE(P)$ and let
$\lambda^*\in\Lambda(x^*)$. Then:
\begin{enumerate}
    \item $g(x) \leq v(x)$ for all $x \in \mathrm{int}(P)$, and $g(x^*) = v(x^*)$.
    \item $c(x^*)=\mathbf0$ and
    $v(x^*)=-\log\V(\MaxIE(P))$; moreover,\\
    $\MaxIE(P)=\E(x^*,Q(x^*)^{-1})$.
    \item $Q(x^*) = \frac{1}{n}S(x^*)^{-1}$.
    \item for every $i\in [m]$ with $\lambda^*_i > 0$, we have $\tilde{a}_i(x^*)^\top S(x^*)^{-1}\tilde{a}_i(x^*) = n$.
\end{enumerate}
\end{lemma}
\begin{proof}
\begin{enumerate}
    \item Use the definition of $v$ and $\lambda^*\in\Lambda(x^*)$.
    \item Write $ L^*:=\mathrm{MinCE}((P-x^*)^\circ)=\E(c(x^*),Q(x^*))$ and $\E^*:=x^*+(L^*)^\circ$.
    The origin lies in the interior of $(P-x^*)^\circ$ and hence in the
    interior of $L^*$.
    Since $L^*\supseteq(P-x^*)^\circ$, polarity reversal and the bipolar
    theorem give $\E^*\subseteq P$, and hence
    $\V(\E^*)\leq\V(\MaxIE(P))$. Conversely,
    $\MaxIE(P)-x^*\subseteq P-x^*$ implies
    $(\MaxIE(P)-x^*)^\circ\supseteq(P-x^*)^\circ$. Thus
    $(\MaxIE(P)-x^*)^\circ$ is feasible for the MinCE problem defining $L^*$,
    so MinCE optimality and
    Lemma~\ref{lemma: useful_lemma} give
    \[
        \V(L^*)\leq\V((\MaxIE(P)-x^*)^\circ)
        =\frac{1}{\V(\MaxIE(P))}.
    \]
    Translation invariance of volume and
    Lemma~\ref{lemma: useful_lemma} applied to $L^*$ yield
    \[
    \begin{aligned}
        1
        &\leq\V(L^*)\V((L^*)^\circ)
        =\V(L^*)\V(\E^*)\\
        &\leq\frac{1}{\V(\MaxIE(P))}\V(\MaxIE(P))
        =1.
    \end{aligned}
    \]
    Therefore equality holds throughout. The equality case in
    Lemma~\ref{lemma: useful_lemma} gives $c(x^*)=\mathbf0$, while
    $\V(\E^*)=\V(\MaxIE(P))$ and uniqueness of the MaxIE give
    $\E^*=\MaxIE(P)$. Applying the centered polar formula to $L^*$ now gives
    $\MaxIE(P)=\E(x^*,Q(x^*)^{-1})$, and
    $\V(L^*)=1/\V(\MaxIE(P))$ gives
    $v(x^*)=-\log\V(\MaxIE(P))$.
    \item Apply \eqref{eq: MinCE_primal_sol} with
    $c(x^*)=\mathbf{0}$.
    \item Repeat the derivation of \eqref{eq: a_S_a_expanded} with
    $c(x^*)=\mathbf{0}$.
\end{enumerate}
\end{proof}

For $h\in\R^n$, choose $\bar t>0$ with
$x^*+th\in\mathrm{int}(P)$ for $|t|<\bar t$, and define
\begin{align}\label{eq: g_directional_derivative}
        g_h(t) := g(x^* + th), \qquad t\in(-\bar t,\bar t).
\end{align}
\begin{lemma}\label{lemma: derivative_of_g}
    Given a direction $h\in \R^n$, define quantities
    \begin{align}\label{eq: g_derivative_components}
        & p_h := S(x^*)^{1/2} h, \quad
        w_i := S(x^*)^{-1/2} \tilde{a}_i(x^*), \notag \\
        & T_h := \sum_{i=1}^m \lambda^*_i (\tilde{a}_i(x^*)^\top h)(w_iw_i^\top).
    \end{align}
    The derivative of $g_h$ at $0$ satisfies $g_h'(0) = 0$, i.e., $\nabla g(x^*) = \mathbf{0}$. In addition, 
    \begin{align*}
            g''_h(0) = (3n-1)\|p_h\|^2 - 2\|T_h\|_F^2.
    \end{align*}
\end{lemma}
\begin{proof}
    Lemma \ref{lemma: differentiation} gives the first derivatives at $t=0$.
    Applying the calculation from Proposition \ref{proposition: v_gradient} to
    \begin{align}\label{eq: g_h_in_t}
        g_h(t) = \frac{1}{2}\log \det S(t) + \frac{n}{2} \log n,
    \end{align}
    its derivative at zero is
    \begin{align*}
        g'_h(0) = & \frac{1}{2}\tr(S(x^*)^{-1}\dot{S})
        = \sum_{i=1}^m \lambda_i^* (\tilde{a}_i(x^*)^\top h)\tilde{a}_i(x^*) ^\top S(x^*)^{-1}\tilde{a}_i(x^*) \\
        = & n\sum_{i=1}^m \lambda_i^* (\tilde{a}_i(x^*)^\top h) =  n \left(\sum_{i=1}^m \lambda_i^* \tilde{a}_i(x^*)\right)^\top h = n c(x^*)^\top h =0,
    \end{align*}
    where Lemma \ref{lemma: facts_about_g} gives the last line.

    For the second derivative, denote those of $(\tilde{a}_i,A,S)$ at $t=0$ by
    $(\ddot{\tilde{a}}_i,\ddot A,\ddot S)$.
    \begin{itemize}
        \item \textbf{Differentiating $\dot{\tilde{a}}_i$.} By \eqref{eq: differentiating_tilde_a}, we have $\tilde{a}'_i(t) = (a_i^\top h)(1 - a_i^\top x^*-t a_i^\top h)^{-2} a_i$, which implies
        \begin{align}\label{eq: second_order_tilde_a}
        \ddot{\tilde{a}}_i= \tilde{a}''_i(0) =  & \frac{2(a_i^\top h)^2}{(1-a_i^\top x^* - ta_i^\top h)^3} a_i \Bigg|_{t=0} \notag \\
        = & 2\left(\frac{a_i^\top h}{1-a_i^\top x^*}\right)^2 \frac{a_i}{1-a_i^\top x^*} = 2(\tilde{a}_i^\top h)^2 \tilde{a}_i.
        \end{align}
        \item \textbf{Differentiating $\dot{A}$.} Differentiating
        $A(t)=\sum_{i=1}^m\lambda_i^*\tilde a_i(t)\tilde a_i(t)^\top$ twice gives
        \begin{align*}
            A''(t) = & \sum_{i=1}^m \lambda_i^*\left(\frac{d}{dt} \tilde{a}'_i(t)\tilde{a}_i(t)^\top + \frac{d}{dt}\tilde{a}_i(t) \tilde{a}'_i(t)^\top \right) \\
            = &  \sum_{i=1}^m \lambda_i^*\left(\tilde{a}''_i(t) \tilde{a}_i(t)^\top + 2\tilde{a}'_i(t)\tilde{a}'_i(t)^\top + \tilde{a}_i(t)\tilde{a}''_i(t)^\top  \right).
        \end{align*}
        At $t=0$, Lemma \ref{lemma: differentiation} and
        \eqref{eq: second_order_tilde_a} give
        \begin{align}\label{eq: second_order_A}
            \ddot{A} = A''(0) =\sum_{i=1}^m \lambda_i^* \left(\ddot{\tilde{a}}_i \tilde{a}_i^\top + 2 \dot{\tilde{a}}_i\dot{\tilde{a}}_i^\top + \tilde{a}_i \ddot{\tilde{a}}_i^\top \right) =  6 \sum_{i=1}^m \lambda_i^* (\tilde{a}_i^\top h)^2 \tilde{a}_i \tilde{a}_i^\top.
        \end{align}
        \item \textbf{Differentiating $\dot{S}$.} Differentiating
        $S=A-cc^\top$ twice and using Lemma \ref{lemma: differentiation},
        \eqref{eq: second_order_A}, and Lemma \ref{lemma: facts_about_g} gives
        \begin{align}
            \ddot{S} = \ddot{A} - (\ddot{c}c^\top + 2\dot{c}\dot{c}^\top + c\ddot{c}^\top) =  6 \sum_{i=1}^m \lambda_i^* (\tilde{a}_i^\top h)^2 \tilde{a}_i \tilde{a}_i^\top - 2(Sh)(Sh)^\top.\label{eq: second_order_S}
        \end{align}
    \end{itemize}
    By \eqref{eq: g_h_in_t} and the fact that $g'_h(t) = \frac{1}{2}\tr(S(t)^{-1}S'(t))$, we have
    \begin{align}
     g_h''(t) = & \frac{1}{2}\frac{d}{dt}\tr(S(t)^{-1}S'(t)) = \frac{1}{2} \tr\left(\frac{d}{dt} S(t)^{-1}S'(t)\right) \notag \\
        = & \frac{1}{2} [\tr((S^{-1})'(t)S'(t)) +  \tr(S^{-1}(t)S''(t))], \label{eq: g_second_derivative}
    \end{align}
    by linearity of trace and the product rule. Differentiating $S^{-1}S=I$ gives
    \begin{alignat*}{2}
        & S(t)^{-1} S(t) = I \quad &&   \\
        \Rightarrow\quad  & (S^{-1})'S + S^{-1}S' = 0 \quad && \text{(differentiate both sides)}\\
        \Rightarrow \quad & (S^{-1})'S S^{-1}+ S^{-1}S'S^{-1}  = 0 \quad && (\text{right multiply by $S^{-1}$}) \\
        \Rightarrow \quad & (S^{-1})' = -S^{-1}S'S^{-1}.
    \end{alignat*}
    Substitution in \eqref{eq: g_second_derivative} gives
    \begin{align}
        g''_h(t)
        = \frac{1}{2}\tr(S^{-1}(t)S''(t)) - \frac{1}{2}\tr(S^{-1}(t)S'(t)S^{-1}(t)S'(t)),\notag
    \end{align}
    and hence, at $t=0$,
    \begin{align}\label{eq: g_second_derivative_2}
        g''_h(0) = \frac{1}2\tr(S(x^*)^{-1} \ddot{S}) - \frac{1}{2}\tr(S(x^*)^{-1}\dot{S}S(x^*)^{-1}\dot{S}).
    \end{align}
    We evaluate the two terms in \eqref{eq: g_second_derivative_2}, writing
    $S=S(x^*)$ and using $S^{-1}=S^{-1/2}S^{-1/2}$ and cyclicity of trace.
    \begin{itemize}
        \item By \eqref{eq: second_order_S} and part (4) of Lemma \ref{lemma: facts_about_g}, we have
        \begin{align}
            S^{-1/2}\ddot{S} S^{-1/2} =& 6 \sum_{i=1}^m \lambda^*_i (\tilde{a}_i^\top h)^2 S^{-1/2} \tilde{a}_i\tilde{a}_i^\top S^{-1/2} - 2S^{-1/2}(Sh)(Sh)^\top S^{-1/2} \notag \\
            = &  6 \sum_{i=1}^m \lambda^*_i (\tilde{a}_i^\top h)^2 w_iw_i^\top - 2p_hp_h^\top  \notag, \\
            \tr(S^{-1}\ddot{S}) = & \tr(S^{-1/2}S^{-1/2}\ddot{S} ) = \tr(S^{-1/2}\ddot{S} S^{-1/2}) \notag \\
            = & 6\sum_{i=1}^m \lambda^*_i(\tilde{a}_i^\top h)^2\|w_i\|^2 - 2\|p_h\|^2  \notag \\
            = & 6n\sum_{i=1}^m \lambda^*_i(\tilde{a}_i^\top h)^2 - 2\|p_h\|^2. \notag
        \end{align}
        Moreover, it holds that 
        \begin{align}
            &\sum_{i=1}^m \lambda^*_i(\tilde{a}_i^\top h)^2  = \sum_{i=1}^m \lambda^*_i(w_i^\top p_h)^2 = p_h^\top \left(\sum_{i=1}^m \lambda^*_i w_iw_i^\top\right) p_h,\notag \\
            &\sum_{i=1}^m \lambda^*_i w_iw_i^\top = S^{-1/2}\left(\sum_{i=1}^m \lambda_i^*\tilde{a}_i\tilde{a}_i^\top \right)S^{-1/2} = S^{-1/2}SS^{-1/2} = I_n. \notag 
        \end{align}
        As a result, we have 
        \begin{align}\label{eq: trace_S_inv_ddot_S}
            \tr(S^{-1}\ddot{S}) = 2(3n-1)\|p_h\|^2.
        \end{align}

        \item By Lemma \ref{lemma: differentiation} and the fact that
        $c=c(x^*)=\mathbf{0}$ from Lemma \ref{lemma: facts_about_g}, we have
        \begin{align*}
            S^{-1/2} \dot{S} S^{-1/2} = 2\sum_{i=1}^m \lambda_i^* (\tilde{a}_i^\top h)w_iw_i^\top = 2T_h,
        \end{align*}
        and hence 
        \begin{align}\label{eq: trace_S_inv_dot_S_sq}
            \tr(S^{-1}\dot{S}S^{-1}\dot{S})  = &   \tr(S^{-1/2}S^{-1/2}\dot{S}S^{-1/2}S^{-1/2}\dot{S}) \notag \\
            = & \tr((S^{-1/2} \dot{S} S^{-1/2} )(S^{-1/2} \dot{S} S^{-1/2})) \notag  \\
            = & \tr((2T_h)^2) = 4\tr(T_h^2) = 4\|T_h\|_F^2.
        \end{align}
    \end{itemize}
    Substituting \eqref{eq: trace_S_inv_ddot_S} and
    \eqref{eq: trace_S_inv_dot_S_sq} into \eqref{eq: g_second_derivative_2} proves the claim.
\end{proof}

The next lemma bounds $\|T_h\|_F$ by $\|p_h\|$.
\begin{lemma}\label{lemma: control_Th_by_ph}
    For every direction $h\in \R^n$, we have  $\|T_h\|^2_F \leq n\|p_h\|^2$.
\end{lemma}
\begin{proof}
    For any symmetric $A$ with $\|A\|_F=1$, Cauchy--Schwarz gives
    \begin{align*}
        |\langle T_h, A\rangle|^2 =  \Bigg|\sum_{i=1}^m \lambda_i^*(\tilde{a}_i^{\top} h)(w_i^\top A w_i)\Bigg|^2\leq  \left(\sum_{i=1}^m \lambda_i^* (\tilde{a}_i^{\top} h)^2\right)\left(\sum_{i=1}^m \lambda_i^* (w_i^\top A w_i)^2\right).
    \end{align*}
    The first factor is
    \begin{align*}
        \sum_{i=1}^m \lambda_i^* (\tilde{a}_i^{\top} h)^2 = h^\top \left(\sum_{i=1}^m \lambda_i^*\tilde{a}_i(x^*)\tilde{a}_i(x^*)^\top\right)h = h^\top S(x^*)h = p_h^\top p_h =\|p_h\|^2.
    \end{align*}
    For the second factor, Cauchy--Schwarz and part 4 of Lemma \ref{lemma: facts_about_g} give, for $i\in[m]$ with $\lambda_i^*>0$,
    \begin{align*}
        (w_i^\top A w_i)^2 \leq \|w_i\|^2\|Aw_i\|^2
        = (\tilde{a}_i(x^*)^\top S(x^*)^{-1}\tilde{a}_i(x^*)) w_i^\top A^2 w_i = nw_i^\top A^2 w_i,
    \end{align*}
    and therefore
    \begin{align*}
        \left(\sum_{i=1}^m \lambda_i^* (w_i^\top A w_i)^2\right)\leq  n \sum_{i=1}^m \lambda_i^* w_i^\top A^2 w_i= n \tr\left(\left(\sum_{i=1}^m\lambda_i^* w_iw_i^\top\right)A^2 \right) = n,
    \end{align*}
    because $\sum_i\lambda_i^*w_iw_i^\top=I_n$ and $\tr(A^2)=\|A\|_F^2=1$.
    Thus it holds that 
    \begin{align*}
        \|T_h\|_F^2 = \max_{A\in\S^n,\|A\|_F=1} |\langle T_h, A\rangle|^2 \leq n\|p_h\|^2.
    \end{align*}
    This completes the proof. 
\end{proof}
We can now prove local quadratic growth at the MaxIE center.
\begin{proposition}\label{prop: local_quadratic_growth}
    The MaxIE center $x^*$ admits an open ball $U_{\mathrm{QG}}$ and a constant $\mu >0$ such that $v(x) -v(x^*) \geq\mu\|x-x^*\|^2$ on $U_{\mathrm{QG}}$.
\end{proposition}
\begin{proof}
    We first consider $n \geq 2$. For every $h\ne\mathbf0$, Lemmas
    \ref{lemma: derivative_of_g} and \ref{lemma: control_Th_by_ph} give
    \begin{align}
     h^\top \nabla^2 g(x^*) h  = g''_h(0) \notag \geq & (3n-1)\|p_h\|^2 - 2n\|p_h\|^2 \notag \\
        = & (n-1)\|p_h\|^2 = \frac{n-1}{n}h^\top Q(x^*)^{-1}h > 0, \notag
    \end{align}
    where $\|p_h\|^2=h^\top S(x^*)h$ and part 3 of Lemma \ref{lemma: facts_about_g}
    give the equality. Thus $\nabla^2g(x^*)\succ0$. By continuity, choose
    $r>0$ and $\mu>0$ so that
    $\bar B_2(x^*,r)\subset\mathrm{int}(P)$ and
    $\nabla^2g(x)\succeq2\mu I_n$ throughout that closed ball, and set
    $U_{\mathrm{QG}}:=B_2(x^*,r)$. The segment from $x^*$ to every
    $x\in U_{\mathrm{QG}}$ then remains in the Hessian-bound neighborhood.
    Lemma \ref{lemma: facts_about_g}, Lemma \ref{lemma: derivative_of_g},
    and Taylor's theorem with integral remainder now give
    \begin{align*}
        v(x)-v(x^*)
        &\geq g(x)-g(x^*) \\
        &=\nabla g(x^*)^\top(x-x^*) \\
        &\quad+\intop\nolimits_0^1(1-\theta)(x-x^*)^\top
        \nabla^2g(x^*+\theta(x-x^*))(x-x^*)\,d\theta \\
        &\geq2\mu\|x-x^*\|^2\intop\nolimits_0^1(1-\theta)\,d\theta
        =\mu\|x-x^*\|^2.
    \end{align*}

    For $n=1$, let $P=[l,u]$ with $u>l$. Then $x^*=(l+u)/2$ and, for $x\in(l,u)$,
    \begin{align*}
        (P-x)^\circ = \Big[-\frac{1}{x-l}, \frac{1}{u-x}\Big].
    \end{align*}
    This interval is its own MinCE, and normalized one-dimensional volume is
    Euclidean length divided by the unit-ball length $2$; hence
    \begin{align*}
        v(x) = \log\left(\frac{1}{2}\left( \frac{1}{u-x} + \frac{1}{x-l}\right)\right) = -\log((u-x)(x-l)) + \log\left(\frac{u-l}{2}\right).
    \end{align*}
    Differentiation gives
    \begin{alignat*}{2}
        v'(x) = & \frac{1}{u-x} - \frac{1}{x-l} \quad &&  \Rightarrow \quad v'(x^*) =  0, \\
        v''(x) =&    \frac{1}{(u-x)^2} + \frac{1}{(x-l)^2} \quad &&               \Rightarrow \quad  v''(x^*) = \frac{8}{(u-l)^2} > 0.
    \end{alignat*}
    Taylor's theorem proves the claim as above.
\end{proof}

\subsection[A Polyak–Łojasiewicz (PL)-type inequality of v around the MaxIE center]{A Polyak–Łojasiewicz (PL)-type inequality of $v$ around the MaxIE center}
Recall that we denote $\tau(x) = c(x)^\top Q(x) c(x)$, where $c(x)$ and $Q(x)$ define
$\mathrm{MinCE}((P-x)^\circ)$. Lemma~\ref{lemma: facts_about_g} gives
$c(x^*)=\mathbf{0}$, and Proposition \ref{proposition: v_gradient} gives
$\nabla v(x)=(n+1)c(x)$. Thus $\tau(x)$ is the squared gradient norm in the
local metric $Q(x)/(n+1)^2$. We first establish descent of $v$ and then locally bound
$v(x)-v(x^*)$ by $\tau(x)$.
\begin{lemma}\label{lemma: descent_of_v}
Recall the potential $v(x)$ defined in \eqref{eq: potential function}. We have
\begin{align*}
    v(\K(x)) - v(x) \leq -\frac{n+1}{2}\log\left(\frac{1}{1-\tau(x)}\right).
\end{align*}
For every nonfixed point $x$, this inequality is strict, i.e., 
$v(\K(x))<v(x)$.
\end{lemma}
\begin{proof}
    Suppress the iteration index in Algorithm \ref{alg: polarity process} and write
    \begin{align*}
        \E^\circ = \E^\circ_k, \quad
        \E^{\circ \circ} = \E'_{k+1}, \quad
        \E^+ = \E_{k+1}.
    \end{align*}
    Translation invariance of volume and Lemma \ref{lemma: useful_lemma} give
    \begin{align}\label{eq: polar_bipolar_viol}
        \log \V(\E^\circ) + \log \V(\E^{+}) = \frac{n+1}{2} \log \left( \frac{1}{1-\tau(x)}\right).
    \end{align}
    Since $\E^+\subseteq P$ is centered at $\K(x)$, the polar of
    $\E^+-\K(x)$ covers $(P-\K(x))^\circ$. Thus
    \begin{align}\label{eq: minCE_optimality}
        v(\K(x)) \leq  \log \V((\E^+-\K(x))^\circ) = -\log \V(\E^+),
    \end{align}
    by \eqref{eq: potential function} and Lemma \ref{lemma: useful_lemma}.
    Combining \eqref{eq: polar_bipolar_viol} and \eqref{eq: minCE_optimality}
    proves the bound.

    Since $Q(x)\succ0$, $\tau(x)=0$ iff $c(x)=\mathbf{0}$, equivalently
    $\K(x)=x$ by Lemma~\ref{lemma: useful_lemma}; hence the bound is strict at
    every nonfixed point.
\end{proof}

\begin{lemma} \label{lemma: perturbed_ellip}
    Let $x^*$ be the center of $\MaxIE(P)$. Define
    \begin{align*}
        \Delta(x) := & x-x^*, \quad \bar{c}(x) := Q(x^*)^{-1}\Delta(x), \\
        \Phi_i(x) := & (\tilde{a}_i(x) - \bar{c}(x))^\top Q(x^*)(\tilde{a}_i(x) - \bar{c}(x)), ~\forall i\in [m].
    \end{align*}
    There exist an open neighborhood $U_{\mathrm{shift}}$ of $x^*$ and $\beta > 0$ such that
    \begin{align}
        \Phi_i(x) \leq 1 + \beta \|x-x^*\|^2, ~\forall i\in [m], \forall x\in U_{\mathrm{shift}}. \notag
    \end{align}
    Moreover, the ellipsoid
    \begin{align}
        \bar{\E}(x) := \left\{y\in \R^n:~ (y-\bar{c}(x))^\top \frac{Q(x^*)}{1 +\beta \|x - x^*\|^2} (y-\bar{c}(x))\leq 1 \right\} \notag
    \end{align}
    contains $(P-x)^\circ$ for every $x\in U_{\mathrm{shift}}$.
\end{lemma}
\begin{proof}
    Lemma~\ref{lemma: facts_about_g} gives $c(x^*)=\mathbf0$, so
    $\mathrm{MinCE}((P-x^*)^\circ)=\E(\mathbf0,Q(x^*))$. Since
    $x^*\in\mathrm{int}(P)$, there is a neighborhood $V$ on which all
    polytope inequalities $a_i^\top x<1$ remain strict and
    $\tilde a_i,\Phi_i$ are smooth. Fix $i\in[m]$
    and $h\in\R^n$, set $x(t)=x^*+th$ on a sufficiently small two-sided
    interval about zero, and use Lemma
    \ref{lemma: differentiation} to obtain
    \begin{align*}
        D\Phi_i(x^*)[h] = & 2 (D\tilde{a}_i(x^*)[h] - D\bar{c}(x^*)[h])^\top Q(x^*)(\tilde{a}_i(x^*) - \bar{c}(x^*)) \\
        = & 2 ( (\tilde{a}_i(x^*)^\top h)\tilde{a}_i(x^*) - Q(x^*)^{-1}h)^\top Q(x^*) \tilde{a}_i(x^*) \\
        = & 2(\tilde{a}_i(x^*)^\top h)(\tilde{a}_i(x^*)^\top Q(x^*)\tilde{a}_i(x^*) - 1).
    \end{align*}
    Define the active set
    \begin{align*}
        \mathcal{A} := \{i\in[m]: \tilde{a}_i(x^*)^\top Q(x^*)\tilde{a}_i(x^*) = 1\},
    \end{align*}
    which is nonempty, since otherwise the ellipsoid could be scaled down. For
    $i\in\mathcal A$, the derivative vanishes, so
    $\nabla\Phi_i(x^*)=\mathbf{0}$; smoothness gives $r_i>0$ with
    $B_2(x^*,r_i)\subseteq V$ and
    \begin{align*}
        H_i := \sup_{y \in B_2(x^*, r_i)} \|\nabla^2\Phi_i(y)\| < +\infty.
    \end{align*}
    Taylor's theorem gives, for $i\in\mathcal A$,
    \begin{align*}
        \Phi_i(x)
        =
        1 + \frac{1}{2}(x-x^*)^\top \nabla^2\Phi_i(\xi)(x-x^*)
        \leq 1 + \frac{H_i}{2}\|x-x^*\|^2,
    \end{align*}
    for some $\xi$ between $x^*$ and $x$.

    If $\mathcal A=[m]$, set $s_{\max}:=1$; otherwise define
    \begin{align*}
        s_{\max} := 1 - \max_{i\notin \mathcal{A}}\tilde{a}_i(x^*)^\top Q(x^*)\tilde{a}_i(x^*) > 0.
    \end{align*}
    For each $i\notin\mathcal A$, continuity gives $r_i>0$ with
    $B_2(x^*,r_i)\subseteq V$ and
    \begin{align*}
        \Phi_i(x) \leq 1 - \frac{s_{\max}}{2}, ~\forall x\in B_2(x^*, r_i).
    \end{align*}
    Define 
    \begin{align*}
        r_{\min} := \min_{i\in [m]} r_i > 0, \quad  U_{\mathrm{shift}} := B_2(x^*, r_{\min}), \quad \beta := \max\left\{ 1, \max_{i\in \mathcal{A}} \frac{H_i}{2}\right\}.
    \end{align*}
    Then for every $x\in U_{\mathrm{shift}}$,
    \begin{align*}
    \Phi_i(x) \leq
    \left \{
    \begin{aligned}
        &1 + \frac{H_i}{2}\|x-x^*\|^2  && \text{if } i \in \mathcal{A} \\
        &1-\frac{s_{\max}}{2} \leq 1 && \text{if } i \notin \mathcal{A}
    \end{aligned}
    \right \} \leq  1 + \beta\|x-x^*\|^2.
\end{align*}
    Finally, $\bar\E(x)\supseteq(P-x)^\circ$ because
    \begin{align*}
        (\tilde{a}_i(x)-\bar c(x))^\top
        \frac{Q(x^*)}{1+\beta\|x-x^*\|^2}
        (\tilde{a}_i(x)-\bar c(x)) 
         = \frac{\Phi_i(x)}{1+\beta\|x-x^*\|^2}\leq1,
        ~i\in[m].
    \end{align*}
    This completes the proof.
\end{proof}

\begin{lemma}\label{lemma: lifted_MinCE_continuity}
For $x\in\operatorname{int}(P)$, the lifted problem
\eqref{eq: lifted_MinCE} has a unique optimizer $\mathcal Q(x)$, and
$\mathcal Q$ is continuous. Consequently, its blocks $G,z,s$, the recovered
maps $c,Q$, and the maps $\tau$ and $\K$ are continuous. Moreover,
\begin{align}\label{eq: self_contained_K_formula}
    \K(x)=x-\frac{Q(x)c(x)}{1-\tau(x)},
\end{align}
and the returned shape
$(1-\tau(x))(Q(x)^{-1}-c(x)c(x)^\top)$ is continuous.
\end{lemma}
\begin{proof}
Set $d:=n+1$ and
$h_i(x):=[\tilde a_i(x)^\top,1]^\top$. Full dimensionality of
$(P-x)^\circ=\operatorname{conv}\{\tilde a_i(x):i\in[m]\}$ implies that the
$h_i(x)$ span $\R^d$. Fix $x_0\in\operatorname{int}(P)$ and set
$W(x):=\sum_i h_i(x)h_i(x)^\top$. Continuity and $W(x_0)\succ0$ give a
neighborhood $U$ of $x_0$ and constants $\eta>0$ and $C_{\mathrm{lift}}\geq1$ such that
$W(x)\succeq\eta I_d$ and $\|h_i(x)\|^2\leq C_{\mathrm{lift}}$ for $i\in[m]$ and $x\in U$.
The summed constraints make the positive-semidefinite relaxation of the feasible
set bounded; it is closed and therefore compact. It contains $C_{\mathrm{lift}}^{-1}I_d$; hence the determinant attains a positive
maximum, giving a positive-definite optimizer.
Strict convexity of $-\log\det$ makes this optimizer unique. Its optimality and \eqref{eq: lifted_MinCE} give
$\det\mathcal Q(x)\geq C_{\mathrm{lift}}^{-d}$ and $\eta\operatorname{tr}\mathcal Q(x) \leq\operatorname{tr}(W(x)\mathcal Q(x))\leq m$; 
so the eigenvalues of $\mathcal Q(x)$ are bounded uniformly above and away
from zero on $U$. In particular, optimizer sequences with parameters in $U$
have positive-definite cluster points.

Let $x_k\to x_0$, write $\mathcal Q_k:=\mathcal Q(x_k)$ and
$\mathcal Q_0:=\mathcal Q(x_0)$, and define
\begin{align*}
    t_k:=\max\left\{1,\max_{i\in[m]}
    h_i(x_k)^\top\mathcal Q_0h_i(x_k)\right\}.
\end{align*}
Then $t_k\to1$ and $\mathcal Q_0/t_k$ is feasible at $x_k$. Along any
convergent subsequence $\mathcal Q_k\to\overline{\mathcal Q}\succ0$,
feasibility passes to the limit, while optimality gives
\begin{align*}
    -\log\det\mathcal Q_k
    \leq-\log\det(\mathcal Q_0/t_k)
    =-\log\det\mathcal Q_0+d\log t_k.
\end{align*}
Therefore $\overline{\mathcal Q}$ is optimal at $x_0$. Strict convexity of
$-\log\det$ makes the optimizer unique, so
$\overline{\mathcal Q}=\mathcal Q_0$. Every cluster point has this value;
hence $\mathcal Q(x_k)\to\mathcal Q(x_0)$.

Block extraction and \eqref{eq: MinCE_primal_sol_from_reformulation} now give
continuity of $G,z,s,c,Q$, and thus of $\tau$; the scalar denominator in the
recovery formula is positive because both $G(x)$ and $Q(x)$ are positive
definite. Finally, $(P-x)^\circ$ contains the origin in its interior and is
contained in $\E(c(x),Q(x))$, so $\tau(x)<1$; applying
Lemma~\ref{lemma: useful_lemma} to the exact polar step gives
\eqref{eq: self_contained_K_formula} and the displayed returned shape.
Their continuity follows from the formulas.
\end{proof}

\begin{proposition}\label{prop: local_PL}
    Let $x^*$ be the center of $\MaxIE(P)$. There exist an open neighborhood $U_{\mathrm{PL}}\subseteq U_{\mathrm{QG}} \cap U_{\mathrm{shift}}$ of $x^*$ and $\sigma > 0$ such that
    \begin{align*}
        v(x) - v(x^*) \leq  \frac{\beta n (n+1)^2}{2\sigma \mu^2}\tau(x),~\forall x\in U_{\mathrm{PL}},
    \end{align*}
    where Lemma \ref{lemma: perturbed_ellip} supplies
    $(\beta,U_{\mathrm{shift}})$ and Proposition
    \ref{prop: local_quadratic_growth} supplies $(\mu,U_{\mathrm{QG}})$.
\end{proposition}
\begin{proof}
    Convexity and Proposition \ref{prop: local_quadratic_growth} give, respectively,
    \begin{align*}
        v(x) -v(x^*) \leq \nabla v(x)^\top (x-x^*) \leq \|\nabla v(x)\|\|x-x^*\|.
    \end{align*}
    \begin{align*}
        v(x) - v(x^*) \geq \mu \|x-x^*\|^2,~\forall x\in U_{\mathrm{QG}}.
    \end{align*}
    For $x\in U_{\mathrm{QG}}\setminus\{x^*\}$, combining them with Proposition
    \ref{proposition: v_gradient} and dividing by $\|x-x^*\|$ gives
    \begin{align*}
        \|c(x)\| = \frac{1}{n+1}\|\nabla v(x)\|\geq \frac{\mu}{n+1} \|x-x^*\|.
    \end{align*}
    The same bound holds at $x=x^*$ because Lemma
    \ref{lemma: facts_about_g} gives $c(x^*)=\mathbf0$.
    Lemma~\ref{lemma: lifted_MinCE_continuity} and $Q(x^*)\succ0$ give $U$
    and $\sigma>0$ with
    $Q(x)\succeq\sigma I_n$ on $U$, hence
    \begin{align}
        \tau(x) = c(x)^\top Q(x)c(x) \geq \sigma \|c(x)\|^2 \geq \frac{\sigma \mu^2}{(n+1)^2}\|x-x^*\|^2, ~\forall x\in U\cap U_{\mathrm{QG}}.  \label{eq: quadratic_lb_tau}
    \end{align}
    Lemma \ref{lemma: perturbed_ellip} supplies $U_{\mathrm{shift}}$, $\beta>0$,
    and the covering ellipsoid
    \begin{align*}
        \bar{\E}(x)
        :=
        \left\{
        y\in\R^n:
        (y-\bar c(x))^\top \frac{Q(x^*)}{1+\beta\|x-x^*\|^2}(y-\bar c(x))\leq 1
        \right\}
    \end{align*}
    for every $x\in U_{\mathrm{shift}}$. MinCE optimality gives
    \begin{align*}
        v(x)\leq \log \V(\bar{\E}(x)), ~ \forall x\in U_{\mathrm{shift}}.
    \end{align*}
    The volume formula gives, successively,
    \begin{align*}
        v(x)
        \leq
        -\frac12\log\det\left(\frac{Q(x^*)}{1+\beta\|x-x^*\|^2}\right).
    \end{align*}
    As a result, we have 
    \begin{align*}
        v(x)-v(x^*)
        &\leq
        -\frac12\log\det\left(\frac{Q(x^*)}{1+\beta\|x-x^*\|^2}\right)
        +\frac12\log\det Q(x^*) \\
        &=
        -\frac12\left(\log\det Q(x^*)-n\log(1+\beta\|x-x^*\|^2)\right)
        +\frac12\log\det Q(x^*) \\
        &=
        \frac n2 \log(1+\beta\|x-x^*\|^2).
    \end{align*}
    Using $\log(1+t)\leq t$ gives
    \begin{align}
        v(x)-v(x^*)
        \leq
        \frac{\beta n}{2}\|x-x^*\|^2,
        \qquad \forall x\in U_{\mathrm{shift}}.
        \label{eq: local_quadratic_upper_v}
    \end{align}
    Combining \eqref{eq: quadratic_lb_tau} and
    \eqref{eq: local_quadratic_upper_v} yields
    \begin{align*}
        v(x)-v(x^*)
        &\leq
        \frac{\beta n (n+1)^2}{2\sigma \mu^2}\tau(x),~ \forall x\in U_{\mathrm{PL}}:=U_{\mathrm{shift}}\cap U\cap U_{\mathrm{QG}}.
    \end{align*}
    Here $U_{\mathrm{PL}}$ is open.
\end{proof}

\subsection{Convergence and iteration complexity of the polarity process}
We first exclude convergence to $\partial P$, i.e., the boundary of $P$.
\begin{lemma}\label{lemma: bounded_iterates}
    Iterates generated by the polarity process stay in the compact set
    \begin{align*}
        \mathcal{L}_{\leq v(x_0)} := \{x \in \mathrm{int}(P):~ v(x) \leq v(x_0)\}.
    \end{align*}
\end{lemma}
\begin{proof}
    By Lemma \ref{lemma: descent_of_v}, the sequence $\{v(x_k)\}_{k\geq0}$ is non-increasing and hence all iterates stay in the sublevel set $\mathcal{L}_{\leq v(x_0)}$. 
    We need to show $\mathcal{L}_{\leq v(x_0)}$ is compact. Since $\mathcal{L}_{\leq v(x_0)}\subseteq P$, we know $\mathcal{L}_{\leq v(x_0)}$ is bounded. 
    Proposition~\ref{proposition: v_gradient} shows that $v$ is differentiable,
    and hence continuous, over $\mathrm{int}(P)$.
    However, this only means $\mathcal{L}_{\leq v(x_0)}$ is closed relative to $\mathrm{int}(P)$: if $s_k \rightarrow s^*$ with $s_k \in \mathcal{L}_{\leq v(x_0)}$ 
    for all $k\in \N$ and $s^* \in \mathrm{int}(P)$, then $s^* \in \mathcal{L}_{\leq v(x_0)}$. It remains to show that any limit point $s^*$ cannot lie on $\partial P$, 
    so that $\mathcal{L}_{\leq v(x_0)}$ is closed in $\R^n$.

    For the purpose of contradiction, suppose $s_k\rightarrow s^*$ and $s^*\in \partial P$. Then by \eqref{eq: polytope}, at least one inequality becomes tight at $s^*$. Passing to a subsequence if necessary, we may assume  $a_{i^*}^\top s_k \rightarrow 1$ for some $i^*\in [m]$ and hence
    $\|\tilde{a}_{i^*}(s_k)\| =\|a_{i^*}/(1-a_{i^*}^\top s_k)\|\rightarrow\infty.$
    In other words, as $s_k\rightarrow \partial P$, the polar $(P-s_k)^\circ$ contains a point whose norm goes to infinity. Since $P$ is bounded, there exists $R>0$ such that $P-s \subseteq \bar{B}_2(\mathbf{0}, R)$ for all $s\in P$, where $\bar{B}_2(\mathbf{0}, R)$ is the closed Euclidean ball centered at the origin with radius $R$. Since polarity reverses inclusion, we have
    $(P-s)^\circ \supseteq \bar{B}_2(\mathbf{0},R)^\circ = \bar{B}_2(\mathbf{0},R^{-1}).$ Since $\mathrm{MinCE}((P-s_k)^\circ) \supseteq (P-s_k)^\circ \supseteq \mathrm{conv}(\tilde{a}_{i^*}(s_k), \bar{B}_2(\mathbf{0}, R^{-1}))$, we have
    \begin{align*}
         v(s_k) = &\log \V (\mathrm{MinCE}((P-s_k)^\circ)) \\
        \geq & \log \V (\mathrm{conv}(\tilde{a}_{i^*}(s_k), \bar{B}_2(\mathbf{0},R^{-1}))) \rightarrow \infty,
    \end{align*}
    contradicting the fact that $s_k \in \mathcal{L}_{\leq v(x_0)}$ for all $k\in \N$.
\end{proof}

Define the potential gap
\begin{align}\label{eq: v_gap}
    \delta_v(x) := v(x) - v(x^*),
\end{align}
where $x^*$ is the center of $\MaxIE(P)$.
\begin{lemma}\label{lemma: unique_minimizer_v}
For the MaxIE center $x^*$,
\[
    v(x)\ge v(x^*), \qquad x\in \operatorname{int}(P),
\]
and equality holds if and only if $x=x^*$. Equivalently,
\[
    \delta_v(x)=0 \quad\Longleftrightarrow\quad x=x^*.
\]
\end{lemma}
\begin{proof}
For $x\in\operatorname{int}(P)$, write $ L_x:=\mathrm{MinCE}((P-x)^\circ)=\E(c(x),Q(x))$ and $\E_x^+:=x+L_x^\circ$.
The origin lies in the interior of $(P-x)^\circ$ and hence of $L_x$.
Moreover, $L_x\supseteq(P-x)^\circ$, so polarity reversal and the bipolar
theorem give $\E_x^+\subseteq P$. Thus $\E_x^+$ is an inscribed ellipsoid and
\begin{align*}
    v(x)+\log\V(\E_x^+)
    =\log\V(L_x)+\log\V(L_x^\circ) =\frac{n+1}{2}\log\frac{1}{1-\tau(x)}\geq0
\end{align*}
by Lemma~\ref{lemma: useful_lemma}. Consequently,
\begin{align*}
    v(x)
    \geq-\log\V(\E_x^+)
    \geq-\log\V(\MaxIE(P))
    =v(x^*),
\end{align*}
where the final equality follows from Lemma~\ref{lemma: facts_about_g}.

If $v(x)=v(x^*)$, equality holds throughout both displays. Hence
$\tau(x)=0$ and $\V(\E_x^+)=\V(\MaxIE(P))$. Uniqueness of the
maximum-volume inscribed ellipsoid gives $\E_x^+=\MaxIE(P)$. Since
$Q(x)\succ0$, $\tau(x)=0$ implies $c(x)=\mathbf0$; Lemma
\ref{lemma: useful_lemma} then shows that $L_x^\circ$ is centered at the
origin. Therefore $\E_x^+$ is centered at $x$, whereas $\MaxIE(P)$ is centered
at $x^*$, and so $x=x^*$. The converse follows directly by taking $x=x^*$.
\end{proof}

\begin{proposition}\label{prop: local_contraction}
    Let $x^*\in \mathrm{int}(P)$ be the center of $\MaxIE(P)$. Recall $\mu$ from Proposition \ref{prop: local_quadratic_growth}, $\beta$ from Lemma \ref{lemma: perturbed_ellip}, and $\sigma$ from Proposition \ref{prop: local_PL}.
    There exists an open neighborhood $U_{\mathrm{loc}} \subseteq U_{\mathrm{PL}}$ of $x^*$ and a constant
    $$q_{\mathrm{loc}} := 1 - \min \left\{\frac{1}{2}, \frac{\sigma \mu^2}{\beta n(n+1)} \right\} \in (0, 1)$$
    such that for every $x \in U_{\mathrm{loc}}$,
    \begin{align*}
        \delta_v(\K(x)) \leq q_{\mathrm{loc}} \delta_v(x)~\text{and}~ \K(x) \in U_{\mathrm{loc}}.
    \end{align*}
\end{proposition}
\begin{proof}
    By Lemma \ref{lemma: descent_of_v}, Proposition \ref{prop: local_PL}, and the fact that $\log(1/(1-t))\geq t$ for $t\in[0,1)$, we have, for all $x\in U_{\mathrm{PL}}$,
    \begin{align*}
        \delta_v(\K(x)) = v(\K(x)) -v(x^*) \leq  & v(x) -v(x^*) - \frac{n+1}{2}\log\left(\frac{1}{1-\tau(x)}\right) \\
        \leq & \delta_v(x) -\frac{n+1}{2}\tau(x) \\
        \leq & \delta_v(x) - \frac{\sigma \mu^2}{\beta n(n+1)} \delta_v(x) \\
        \leq & q_{\mathrm{loc}}\delta_v(x).
    \end{align*}
    Choose $r>0$ with $\bar{B}_2(x^*,r)\subseteq U_{\mathrm{PL}}$ and set
    $$m_r := \min_{z} \{\delta_v(z): \|z-x^*\| = r\}. $$
    Continuity and Lemma \ref{lemma: unique_minimizer_v} give $m_r>0$. Define
    \begin{align*}
        U_\mathrm{loc} := \{x \in {B}_2(x^*, r): \delta_v(x) < m_r\}.
    \end{align*}
    If $x\in U_\mathrm{loc}$, then $\delta_v(\K(x))<m_r$. If
    $\K(x) \notin U_\mathrm{loc}$, then necessarily
    $\|x^*-\K(x)\|\geq r$. Define
    $$t := r/\|x^*-\K(x)\|\in (0, 1],~ z :=(1-t)x^* +t \K(x).$$
    Then $\|x^*-z\|=r$, while convexity gives
    \begin{align*}
        t\delta_v(\K(x)) + (1-t) \delta_v(x^*) \geq \delta_v(z)
        \Rightarrow  \delta_v(\K(x)) \geq \frac{\delta_v(z)}{t} \geq \frac{m_r}{t} \geq m_r,
    \end{align*}
    a contradiction. So we must have $\K(x)\in U_{\mathrm{loc}}$.
\end{proof}

We now prove trajectory-wise global linear contraction and an iteration bound
whose rate is not evaluated from $m,n,$ and $R$.
\begin{theorem}\label{thm: polarity_process_convergence_rate}
    For every $x_0\in \mathrm{int}(P)$, there exists $q := q(x_0)\in (0,1)$ such that iterates generated by the polarity process satisfy
    \begin{align}
        \delta_v(x_{k+1}) \leq q \delta_v(x_k), \qquad k\geq0. \notag
    \end{align}
    Consequently, the following results hold.
    \begin{enumerate}
        \item\textbf{(Asymptotic convergence)} The polarity process globally converges to $\MaxIE(P)$. More specifically, for any $x_0 \in \mathrm{int}(P)$, we have
        \begin{align}
            & x_k \rightarrow x^*, \notag \\
            & E_{k+1} := (1-\tau(x_k))(Q(x_k)^{-1} - c(x_k)c(x_k)^\top)\rightarrow Q(x^*)^{-1}, \notag
        \end{align}
        where $\E(x^*, Q(x^*)^{-1})=\MaxIE(P)$.
        \item \textbf{(Iteration complexity on volume ratio)} Let $\gamma \in (0, 1)$ and
        \begin{align*}
            N_\gamma^{\mathrm{ex}}
            = \left\lceil
            \frac{\log^+ (\delta_v(x_0)/\log(1/\gamma))}{\log(1/q)}
            \right\rceil.
        \end{align*}
        Within $N_\gamma^{\mathrm{ex}}+1$ iterations, the polarity process
        produces a $\gamma$-maximal ellipsoid of $P$:
        \begin{align}
            \frac{\Vol( \E_{N_\gamma^{\mathrm{ex}}+1})}
            {\Vol(\MaxIE(P))} \geq \gamma.
        \end{align}
    \end{enumerate}
\end{theorem}
\begin{proof}
    Lemma \ref{lemma: bounded_iterates} places all iterates in the compact set
    $\mathcal{L}_{\leq v(x_0)}$. Since $U_{\mathrm{loc}}$ is open,
    $U_{\mathrm{far}}:=\mathcal{L}_{\leq v(x_0)}\setminus U_{\mathrm{loc}}$ is
    compact. If it is empty, take $q=q_{\mathrm{loc}}$. Otherwise, Lemma
    \ref{lemma: descent_of_v} gives, for every $x\in U_{\mathrm{far}}$,
    \begin{align*}
        \frac{n+1}{2}\log \frac{1}{1-\tau(x)}
        \leq v(x) - v(\K(x))
        \leq v(x)- v(x^*) = \delta_v(x).
    \end{align*}
    Since $x^*\in U_{\mathrm{loc}}$, Lemma
    \ref{lemma: unique_minimizer_v} gives $\delta_v(x)>0$ on
    $U_{\mathrm{far}}$. If $c(x)=\mathbf0$ there, Proposition
    \ref{proposition: v_gradient} gives $\nabla v(x)=\mathbf0$; convexity would
    make $x$ a global minimizer of $v$, contradicting Lemma
    \ref{lemma: unique_minimizer_v}. Thus $c(x)\ne\mathbf0$ and $\tau(x)>0$.
    Therefore
    \begin{align*}
        \Psi(x) := \frac{\frac{n+1}{2}\log \frac{1}{1-\tau(x)}}{\delta_v(x)}
    \end{align*}
    is continuous with values in $(0,1]$ on the compact set
    $U_{\mathrm{far}}$: continuity of $\tau$ follows from Lemma
    \ref{lemma: lifted_MinCE_continuity}, while Proposition
    \ref{proposition: v_gradient} makes $\delta_v$ continuous. Hence
    \begin{align*}
        \Psi^* := \min_{x\in U_{\mathrm{far}}} \Psi(x) > 0,
    \end{align*}
    and the descent bound gives
    \begin{align*}
        \delta_v(\K(x))
        \leq
        \delta_v(x)-\Psi(x)\delta_v(x)
        \leq (1-\Psi^*)\delta_v(x).
    \end{align*}
    Set $q_{\mathrm{far}}:=\max\{0,1-\Psi^*\}<1$ and
    $q:=\max\{q_{\mathrm{loc}},q_{\mathrm{far}}\}\in(0,1)$. Both cases give
    $\delta_v(x_{k+1})\leq q\delta_v(x_k)$. Compactness and contraction make every limit point satisfy $\delta_v=0$,
    hence equal $x^*$ by Lemma \ref{lemma: unique_minimizer_v}; thus $x_k\to x^*$.
    The returned-shape continuity in Lemma
    \ref{lemma: lifted_MinCE_continuity} and Lemma
    \ref{lemma: facts_about_g} give
    $E_{k+1}\to Q(x^*)^{-1}$, and the latter lemma identifies
    $\E(x^*,Q(x^*)^{-1})$ with $\MaxIE(P)$.

    Finally, Lemmas \ref{lemma: useful_lemma} and
    \ref{lemma: facts_about_g}, the definition of $v$, the contraction,
    and the stated value of $N_\gamma^{\mathrm{ex}}$ give
    \begin{align*}
        & \log \left(\frac{\V(\MaxIE(P))}
        {\V(\E_{N_\gamma^{\mathrm{ex}}+1})}\right) \\
        = & \log \V(\MaxIE(P))
        - \log \V(\E_{N_\gamma^{\mathrm{ex}}+1}) \\
        = & -\log \V(\mathrm{MinCE}((P-x^*)^\circ)) \\
        &\quad + \log \V(\mathrm{MinCE}((P-x_{N_\gamma^{\mathrm{ex}}})^\circ))
        - \frac{n+1}{2}\log \frac{1}{1-\tau(x_{N_\gamma^{\mathrm{ex}}})}\\
        =  & v(x_{N_\gamma^{\mathrm{ex}}}) - v(x^*)
        - \frac{n+1}{2}\log \frac{1}{1-\tau(x_{N_\gamma^{\mathrm{ex}}})} \\
        \leq & \delta_v(x_{N_\gamma^{\mathrm{ex}}})
        \leq q^{N_\gamma^{\mathrm{ex}}}\delta_v(x_0)
        \leq \log (1/\gamma).
    \end{align*}
    This completes the proof. 
\end{proof}

\begin{corollary}\label{corollary: bounded radius ratio}
    Suppose there exist $r_1 > 0$ and $r_2 > 0$ such that
    \begin{align*}
        \bar{B}_2(\mathbf{0}, r_1) \subseteq P \subseteq \bar{B}_2(\mathbf{0}, r_2).
    \end{align*}
    Let $\gamma\in(0,1)$, $R=r_2/r_1\geq1$, and $x_0=\mathbf{0}$. Then
    \begin{align*}
            N_\gamma^{\mathrm{ex}} \leq \left\lceil
            \frac{\log^+\!\left(n\log R/\log(1/\gamma)\right)}{\log(1/q)}
            \right\rceil.
        \end{align*}
\end{corollary}
\begin{proof}
    When $R=1$, the two balls coincide, so the sandwich forces $P$ to equal
    that ball and $x_0$ is already the center of $\MaxIE(P)$; the convention
    $\log^+(0)=0$ gives $N_\gamma^{\mathrm{ex}}=0$. For $R>1$, proceed as follows.
    Polarity reversal gives
    \begin{align*}
         (P-x_0)^\circ  = P^\circ \subseteq \bar{B}_2(\mathbf{0}, r_1)^\circ = \bar{B}_2(\mathbf{0}, 1/r_1).
    \end{align*}
    As a result, we have 
    \begin{align*}
        v(x_0) \leq &\log \V(\bar{B}_2(\mathbf{0}, 1/r_1)) = -n \log r_1.
    \end{align*}
    Lemma \ref{lemma: facts_about_g} and
    $P\subseteq\bar B_2(\mathbf{0},r_2)$ give
    \begin{align*}
        v(x^*) = & - \log \V(\MaxIE(P)) \\
        \geq & - \log \V( \bar{B}_2(\mathbf{0}, r_2) ) = -n \log r_2.
    \end{align*}
    As a result, we have $\delta_v(x_0) = v(x_0) - v(x^*) \leq n \log R$, substituting which in Theorem \ref{thm: polarity_process_convergence_rate} proves the claim. 
\end{proof}

\section{Inexact polarity process and conditional arithmetic estimates}\label{sec:inexact-polarity-process}
Practical implementations compute $\mathrm{MinCE}((P-x_k)^\circ)$ inexactly.
For $x\in\mathrm{int}(P)$, we adopt the lifted convex formulation
\eqref{eq: lifted_MinCE}:
\begin{align*}
    \min_{\mathcal{Q} \in \S^{n+1}_{++}} \quad & -\log \det \mathcal{Q} \\
    \mathrm{s.t.}\quad &
    \begin{bmatrix}
        \tilde{a}_i(x)\\ 1
    \end{bmatrix}^\top
    \mathcal{Q}
    \begin{bmatrix}
        \tilde{a}_i(x)\\ 1
    \end{bmatrix} \leq 1, ~\forall i\in[m].
\end{align*}
Section \ref{subsec: inexact_pp} analyzes an inexact solution oracle for
\eqref{eq: lifted_MinCE}, while Section
\ref{subsec: arith} applies existing MinCE solvers to derive conditional arithmetic estimates.

\subsection{Inexact polarity process}\label{subsec: inexact_pp}
A feasible lifted solution yields a covering ellipsoid in $\R^n$ whose polar is
inscribed in $P$.
\begin{lemma}\label{lemma: inexact_basic_properties}
    Let $x\in \mathrm{int}(P)$ and
    \begin{align}
       \widehat{\mathcal{Q}}(x) := \begin{bmatrix}
            \widehat{G}(x) & \widehat{z}(x)\\\widehat{z}(x)^\top & \widehat{s}(x)
        \end{bmatrix}
        \in \S^{n+1}_{++} \notag
    \end{align}
    be a feasible solution to \eqref{eq: lifted_MinCE}. Define
    \begin{align}
        \widehat{\beta}(x) := 1 - \widehat{s}(x) + \widehat{z}(x)^\top \widehat{G}(x)^{-1}\widehat{z}(x). \notag
    \end{align}
    Then:
    \begin{enumerate}
        \item  We have $\widehat{\beta}(x) > 0$, and $(P-x)^\circ \subseteq \E(\widehat{c}, \widehat{Q})$, where
            \begin{align}
                \widehat{c}(x): = -\widehat{G}^{-1}(x) \widehat{z}(x), \quad \widehat{Q}(x) := \frac{\widehat{G}(x)}{\widehat{\beta}(x)}. \notag
            \end{align}
        \item It holds that $\wh{c}(x)^\top \wh{Q}(x) \wh{c}(x) < 1$ and $1 - \wh{s}(x) > 0$.
        \item It holds that $\E(\wh{\K}(x), \wh{E}(x)) \subseteq P$, where
        \begin{align*}
            \wh{\K}(x) := & x + \frac{\wh{z}(x)}{1-\wh{s}(x)}, \\
            \wh{E}(x) := & (1-\wh{s}(x)) \left(\wh{G}(x)^{-1} - \frac{\wh{G}(x)^{-1}\wh{z}(x)\wh{z}(x)^{\top}\wh{G}(x)^{-1}}{\wh{\beta}(x)}\right);
        \end{align*}
    \end{enumerate}
\end{lemma}
\begin{proof}
    \begin{enumerate}
        \item Intersect the lifted ellipsoid with the height-one hyperplane:
        $$
            \E_{\mathrm{intersection}} = \left \{u\in \R^n:~  \begin{bmatrix}
                u \\ 1
            \end{bmatrix}
            ^\top \wh{\mathcal{Q}}(x)
            \begin{bmatrix}
                u \\ 1
            \end{bmatrix} \leq 1\right \},
        $$
        Completing the square gives
        \begin{align}
            \begin{bmatrix}
                u \\ 1
            \end{bmatrix}
            ^\top \wh{\mathcal{Q}}(x)
            \begin{bmatrix}
                u \\ 1
            \end{bmatrix}
            =  & u^\top \wh{G}(x)u + 2\wh{z}(x)^\top u + \wh{s}(x) \notag \\
            = & (u+\wh{G}(x)^{-1}\wh{z}(x))^\top\wh{G}(x)
            (u+\wh{G}(x)^{-1}\wh{z}(x)) \notag \\
            &\quad +\wh{s}(x)-\wh{z}(x)^\top\wh{G}(x)^{-1}\wh{z}(x).\notag
        \end{align}
        Thus membership is equivalent to
        $(u-\wh c(x))^\top\wh G(u-\wh c(x))\leq\wh\beta(x)$. The section
        contains the full-dimensional hull of the lifted vertices, so
        $\wh\beta(x)>0$ and the section is $\E(\wh c(x),\wh Q(x))$ containing
        $(P-x)^\circ$.
        \item Because $\mathbf{0}\in\mathrm{int}((P-x)^\circ)$,
        $\wh c(x)^\top\wh Q(x)\wh c(x)<1$. The definitions give
        $$
            1-\wh{s}(x) = \wh{\beta}(x) - \widehat{z}(x)^\top \widehat{G}(x)^{-1}\widehat{z}(x) = \wh{\beta}(x) (1 - \wh{c}(x)^\top \wh{Q}(x)\wh{c}(x)) >0.
        $$
        \item Lemma \ref{lemma: useful_lemma} gives the polar center and shape:
        \begin{align*}
            & -\frac{\wh{Q}(x)\wh{c}(x)}{1-\wh{c}(x)^\top \wh{Q}(x)\wh{c}(x)}
            = \frac{\wh{z}(x)}{1-\wh{s}(x)}, \\
            &(1-\wh{c}(x)^\top \wh{Q}(x)\wh{c}(x))
            \left(\wh{Q}(x)^{-1}-\wh{c}(x)\wh{c}(x)^\top\right)\\
            = & 
            (1-\wh{s}(x))
            \left(
                \wh{G}(x)^{-1}
                -
                \frac{\wh{G}(x)^{-1}\wh{z}(x)\wh{z}(x)^\top\wh{G}(x)^{-1}}{\wh{\beta}(x)}
            \right),
        \end{align*}
        and hence $\E(\wh\K(x),\wh E(x))$ is this polar shifted by $x$ and is
        contained in $P$.
    \end{enumerate}
\end{proof}
This covering yields an inscribed ellipsoid and defines the following inexact
oracle.
\begin{definition}\label{def: inexact oracle}
    For $x\in\operatorname{int}(P)$, let $\mathcal Q(x)$ solve \eqref{eq: lifted_MinCE}.
    An inexact MinCE oracle of tolerance $\epsilon\geq0$ returns a feasible
    \begin{align}
        \widehat{\mathcal{Q}}(x) := \begin{bmatrix}
            \widehat{G}(x) & \widehat{z}(x) \\
            \widehat{z}(x)^\top & \widehat{s}(x)
        \end{bmatrix}
        \in \S^{n+1}_{++} \notag
    \end{align}
    satisfying
    \begin{align*}
            -\log \det \widehat{\mathcal{Q}}(x)   \leq  -\log \det \mathcal{Q}(x) + \epsilon.
    \end{align*}
    We write $\wh{\mathcal Q}(x)=\mathrm{MinCE}_{\epsilon}((P-x)^\circ)$.
\end{definition}

Algorithm \ref{alg: inexact_polarity_process} displays the inexact iteration for a supplied tolerance $\epsilon$. 
Theorem \ref{thm: inexact_fixed_tolerance} gives an estimate of $\epsilon$ given the target volume ratio $\gamma$. 
\begin{algorithm}[H]
\caption{Inexact Polarity Process}
\label{alg: inexact_polarity_process}
\begin{algorithmic}[1]
\State \textbf{Input:} $x_0\in\mathrm{int}(P)$, and tolerance $\epsilon\geq0$; 
\For{$k=0,1,2\ldots,$}
\State invoke an inexact MinCE oracle with tolerance $\epsilon$:
\begin{align}
        \widehat{\mathcal{Q}}(x_k) = \begin{bmatrix}
            \widehat{G}(x_k) & \widehat{z}(x_k) \\
            \widehat{z}(x_k)^\top & \widehat{s}(x_k)
        \end{bmatrix}
        \gets \mathrm{MinCE}_{\epsilon}((P-x_k)^\circ); \notag
    \end{align}
\State compute the primal inscribed ellipsoid $\E(x_{k+1}, E_{k+1})$ with
\begin{align}
    x_{k+1} \gets & \wh{\K}(x_k) = x_k + \frac{\wh{z}(x_k)}{1-\wh{s}(x_k)}, \notag \\
    E_{k+1} \gets & \wh{E}(x_k) = (1-\wh{s}(x_k))  \left(\wh{G}(x_k)^{-1} - \frac{\wh{G}(x_k)^{-1}\wh{z}(x_k)\wh{z}(x_k)^{\top}\wh{G}(x_k)^{-1}}{1-\wh{s}(x_k) + \wh{z}(x_k)^\top \wh{G}(x_k)^{-1}\wh{z}(x_k)}\right)\notag;
\end{align}
\EndFor
\end{algorithmic}
\end{algorithm}

Because inexact MinCE solutions need not satisfy Lemma
\ref{lemma: descent_of_v}, Algorithm \ref{alg: inexact_polarity_process} may
leave $\mathcal{L}_{\leq v(x_0)}$; we therefore enlarge it to
\begin{align}\label{eq: enlarged_sublevel_set}
    \mathcal{L}_{\leq v(x_0)+1} := \{ x \in \operatorname{int}(P): v(x) \leq v(x_0) + 1\}.
\end{align}
For small $\epsilon$, iterates remain here. This compact region makes the exact
polarity step uniformly contractive and provides a stable reference for the
perturbation estimates.
\begin{lemma}\label{lemma: inexact_sublevel_set}
    The sublevel set $\mathcal{L}_{\leq v(x_0)+1}$ is compact. Moreover, there exists $q_0\in (0,1)$ such that
        \begin{equation}\label{eq: inexact_q0}
            \delta_v(\K(x))\le q_0\delta_v(x), ~ \forall x\in \mathcal{L}_{\leq v(x_0)+1}.
        \end{equation}
\end{lemma}
\begin{proof}
    The boundary-blowup proof of Lemma \ref{lemma: bounded_iterates} applies
    verbatim to this finite sublevel and proves compactness. Repeating the
    proof of Theorem \ref{thm: polarity_process_convergence_rate} with
    $U_\mathrm{far}:=\mathcal{L}_{\leq v(x_0)+1}\setminus U_{\mathrm{loc}}$ gives
    $q_0$: Lemma \ref{lemma: unique_minimizer_v} keeps $\delta_v$ positive off
    $x^*$, and compactness makes the far-region ratio positive.
\end{proof}

For $x\in\mathrm{int}(P)$, write the optimizer of \eqref{eq: lifted_MinCE} as
\begin{align}
    \mathcal{Q}(x)=
    \begin{bmatrix}
        G(x) & z(x)\\
        z(x)^\top & s(x)
    \end{bmatrix},
\end{align}
and the exact polarity update has the form
\begin{align*}
    \K(x):= & x+\frac{z(x)}{1-s(x)}, \\
    E(x):= & (1-s(x))
    \left(
        G(x)^{-1}
        -
        \frac{G(x)^{-1}z(x)z(x)^\top G(x)^{-1}}{\beta(x)}
    \right), \\
    \text{where}~ \beta(x) := & 1-s(x) + z(x)^\top G(x)^{-1}z(x).
\end{align*}

We next isolate a neighborhood where the inexact update is well-defined and
Lipschitz in the lifted-solution error.
\begin{lemma}\label{lemma: inexact_block_domain_lipschitz}
    Let
    $$
        \widehat{\mathcal{Q}}(x)=\begin{bmatrix}\widehat{G}(x)&\widehat{z}(x)\\\widehat{z}(x)^\top&\widehat{s}(x)\end{bmatrix}\in \S^{n+1}_{++}
    $$
    be any feasible point of \eqref{eq: lifted_MinCE}, and let $\widehat{\beta}(x)$, $\widehat{\K}(x)$, $\widehat{E}(x)$ be defined as in Lemma \ref{lemma: inexact_basic_properties}.
    Then the following claims hold.
    \begin{enumerate}
        \item The four quantities defined below are strictly positive:
        \begin{align*}
            \theta_0 := & \min_{x\in \mathcal{L}_{\leq v(x_0)+1} }
            \lambda_{\min}(\mathcal Q(x))>0,\\
            \mu_0 := & \min_{x\in \mathcal{L}_{\leq v(x_0)+1} }\lambda_{\min}(G(x)) > 0, \\
            \quad
            \sigma_0 := & \min_{x\in \mathcal{L}_{\leq v(x_0)+1} }(1-s(x))> 0 ,\\
            \quad
            \beta_0 := & \min_{x\in \mathcal{L}_{\leq v(x_0)+1} }\beta(x)> 0 .
        \end{align*}
        \item There exists $\rho_0> 0$ such that for all $x\in \mathcal{L}_{\leq v(x_0)+1}$ with
        \begin{equation}\label{eq: inexact_rho0}
            \|\widehat{\mathcal{Q}}(x)-\mathcal{Q}(x)\|_F\le \rho_0,
        \end{equation}
        one has
        \begin{align}\label{eq: inexact_block_domain_bounds}
            \lambda_{\min}(\widehat{\mathcal Q}(x))\ge \frac{\theta_0}{2},~
            \lambda_{\min}(\widehat G(x))\ge \frac{\mu_0}{2},~
            1-\widehat s(x)\ge \frac{\sigma_0}{2},~
            \widehat\beta(x)\ge \frac{\beta_0}{2}.
        \end{align}
        As a result, Lemma~\ref{lemma: inexact_basic_properties} makes
        $\widehat{\K}(x)$ and $\widehat{E}(x)$ well-defined.
        \item There exist $L_{\mathrm{ctr}} >0$ and $L_{\mathrm{vol}}>0$ such that for all $x\in \mathcal{L}_{\leq v(x_0)+1}$ with $\|\widehat{\mathcal{Q}}(x)-\mathcal{Q}(x)\|_F\le \rho_0$,
        one has
        \begin{align}
            & \|\widehat{\K}(x)-\K(x)\|
            \leq
            L_{\mathrm{ctr}}\|\widehat{\mathcal{Q}}(x)-\mathcal{Q}(x)\|_F,  \label{eq: inexact_lipschitz_constants_center}  \\
            & \left|
                \log\V(\E(\widehat{\K}(x),\widehat{E}(x)))
                -
                \log\V(\E(\K(x),E(x)))
            \right|
            \leq
            L_{\mathrm{vol}}\|\widehat{\mathcal{Q}}(x)-\mathcal{Q}(x)\|_F. \label{eq: inexact_lipschitz_constants_vol}
        \end{align}
    \end{enumerate}
\end{lemma}
\begin{proof}
    \begin{enumerate}
        \item Lemma~\ref{lemma: lifted_MinCE_continuity} gives continuity of
        $\mathcal{Q},G,z,$ and $s$. Since $G(x)\succ0$, the definition of
        $\beta$ also makes $\beta$ continuous. Compactness of
        $\mathcal{L}_{\leq v(x_0)+1}$, together with
        $\mathcal Q(x)\succ0$, $G(x)\succ0$,
        $1-s(x)>0$, and $\beta(x)>0$ from Lemma
        \ref{lemma: inexact_basic_properties}, gives
        $\theta_0,\mu_0,\sigma_0,\beta_0>0$.

        \item Define the exact graph
        \[
            \mathcal{G}_0:=\{(x,\mathcal{Q}(x)):x\in \mathcal{L}_{\leq v(x_0)+1} \}.
        \]
        This graph is compact.
        The mappings
        \[
            (x,\widehat{\mathcal{Q}})\mapsto \lambda_{\min}(\widehat{\mathcal Q}),\quad
            (x,\widehat{\mathcal{Q}})\mapsto \lambda_{\min}(\widehat{G}),\quad
            (x,\widehat{\mathcal{Q}})\mapsto 1-\widehat{s},\quad \text{and}~
            (x,\widehat{\mathcal{Q}})\mapsto \widehat{\beta}
        \]
        are continuous on $\widehat{\mathcal{Q}}\in\S^{n+1}_{++}$. On
        $\mathcal{G}_0$, they have lower bounds
        $\theta_0,\mu_0,\sigma_0,$ and $\beta_0$, respectively.
        By continuity and compactness of $\mathcal G_0$, there exists
        $\rho_0>0$ such that the claimed bounds hold throughout the entire
        ambient $\rho_0$-tube around $\mathcal G_0$ in
        $\mathcal L_{\leq v(x_0)+1}\times\S^{n+1}$. In particular,
        $\widehat{\K}(x)$ and $\widehat{E}(x)$ are defined.

        \item Define
        \[
            \mathcal{T}_0:=\left\{(x,\wh{\mathcal{Q}}):
            x\in \mathcal{L}_{\leq v(x_0)+1},\ \wh{\mathcal{Q}}\in\S^{n+1},\
            \|\wh{\mathcal{Q}} -\mathcal{Q}(x)\|_F\le \rho_0\right\}.
        \]
        Continuity of $\mathcal Q(x)$ makes $\mathcal{T}_0$ compact. The margin
        bounds in the previous part place $\mathcal{T}_0$ in the open region where
        $\wh{\mathcal{Q}} \succ0$, $\wh{G} \succ 0$, $1-\wh{s} >0$, and $\wh{\beta} >0$.
        The mappings from $(x,\wh{\mathcal{Q}})$ to $\wh{\K}(x)$ and  $\log\V(\E(\widehat{\K}(x),\widehat{E}(x)))$
        defined by the block formulas are smooth on this region. Their
        derivatives with respect to $\widehat{\mathcal{Q}}$ are therefore bounded on the
        compact $\mathcal{T}_0$. As a result, the two Lipschitz estimates \eqref{eq: inexact_lipschitz_constants_center} and \eqref{eq: inexact_lipschitz_constants_vol} hold near $(x, \mathcal{Q}(x))$.
    \end{enumerate}
\end{proof}

\begin{lemma}\label{lemma: inexact_oracle_recovery}
    Define
    \begin{align}\label{eq: inexact_lambda0}
        W(x):=
        \sum_{i=1}^m
        \begin{bmatrix}
            \tilde{a}_i(x)\\ 1
        \end{bmatrix}
        \begin{bmatrix}
            \tilde{a}_i(x)\\ 1
        \end{bmatrix}^{\top},
        ~
        \lambda_0:=\min_{x\in \mathcal{L}_{\leq v(x_0)+1}}\lambda_{\min}(W(x))>0,
    \end{align}
    and let
    \[
        B_0:=\frac{\sqrt{2}\,mL_{\mathrm{ctr}}}{\lambda_0},
        \qquad
        D_0:=\frac{\sqrt{2}\,mL_{\mathrm{vol}}}{\lambda_0}.
    \]
    Then the following claims hold.
    \begin{enumerate}
        \item For every $x\in \mathcal{L}_{\leq v(x_0)+1}$ and every feasible $\widehat{\mathcal{Q}}(x)$ of \eqref{eq: lifted_MinCE}, we have
        \[
            \lambda_{\max}(\widehat{\mathcal{Q}}(x))\le \frac{m}{\lambda_0}.
        \]
        \item If, in addition, $\widehat{\mathcal{Q}}(x)$ is $\epsilon$-optimal for \eqref{eq: lifted_MinCE}, namely,
        \[
            -\log \det \widehat{\mathcal{Q}}(x)\le -\log \det \mathcal{Q}(x)+\epsilon,
        \]
        then
        \[
            \|\widehat{\mathcal{Q}}(x)-\mathcal{Q}(x)\|_F
            \le \frac{\sqrt{2}\,m}{\lambda_0}\sqrt{\epsilon}.
        \]
        \item If $\epsilon\le \lambda_0^2\rho_0^2/(2m^2)$, then $\widehat{\K}(x)$ and $\widehat{E}(x)$ are well-defined and
        \begin{align*}
            & \|\widehat{\K}(x)-\K(x)\|
            \le B_0\sqrt{\epsilon}, \\
            & \left|
                \log\V(\E(\widehat{\K}(x),\widehat{E}(x)))
                -
                \log\V(\E(\K(x),E(x)))
            \right|
            \le D_0\sqrt{\epsilon}.
        \end{align*}
    \end{enumerate}
\end{lemma}
\begin{proof}
    \begin{enumerate}
        \item Full dimensionality makes the lifted vertices span $\R^{n+1}$,
        so $W(x)\succ0$. Continuity and compactness give $\lambda_0>0$. For
        feasible $\widehat{\mathcal Q}(x)$, summing the constraints gives
        \[
            \tr(W(x)\widehat{\mathcal{Q}}(x))
            =
            \sum_{i=1}^m
            \begin{bmatrix}
                \tilde{a}_i(x)\\ 1
            \end{bmatrix}^{\top}
            \widehat{\mathcal{Q}}(x)
            \begin{bmatrix}
                \tilde{a}_i(x)\\ 1
            \end{bmatrix}
            \le m.
        \]
        Since $W(x)\succeq \lambda_0 I$,
        \[
            \lambda_0\tr(\widehat{\mathcal{Q}}(x))
            \le
            \tr(W(x)\widehat{\mathcal{Q}}(x))
            \le m,
        \]
        hence
        \[
            \lambda_{\max}(\widehat{\mathcal{Q}}(x))
            \le
            \tr(\widehat{\mathcal{Q}}(x))
            \le
            \frac{m}{\lambda_0}.
        \]

        \item Let $f(\mathcal{Q}):=-\log \det \mathcal{Q}$. Its Hessian is
        \[
            \nabla^2 f(\mathcal{Q})[U,U]
            =
            \tr(\mathcal{Q}^{-1}U\mathcal{Q}^{-1}U).
        \]
        Part 1 gives $\mathcal{Q}^{-1}\succeq(\lambda_0/m)I$ for every feasible
        $\mathcal Q$, and therefore
        \[
            \nabla^2 f(\mathcal{Q})[U,U]
            \ge
            \frac{\lambda_0^2}{m^2}\|U\|_F^2.
        \]
        Thus $f$ is uniformly $\lambda_0^2/m^2$-strongly convex. Optimality of
        $\mathcal Q(x)$ gives
        \[
            \langle \nabla f(\mathcal{Q}(x)),\widehat{\mathcal{Q}}(x)-\mathcal{Q}(x)\rangle \ge 0,
        \]
        for all feasible $\widehat{\mathcal{Q}}(x)$,  and strong convexity further implies that 
        \[
            f(\widehat{\mathcal{Q}}(x))-f(\mathcal{Q}(x))
            \ge
            \frac{\lambda_0^2}{2m^2}\|\widehat{\mathcal{Q}}(x)-\mathcal{Q}(x)\|_F^2,
        \]
        which, with the assumed objective gap, proves the claim.

        \item If $\epsilon\le \lambda_0^2\rho_0^2/(2m^2)$, then the second part gives
        \[
            \|\widehat{\mathcal{Q}}(x)-\mathcal{Q}(x)\|_F
            \le
            \frac{\sqrt{2}\,m}{\lambda_0}\sqrt{\epsilon}
            \le
            \rho_0.
        \]
        Lemma \ref{lemma: inexact_block_domain_lipschitz} now proves the two bounds.
    \end{enumerate}
\end{proof}

Combining exact-map contraction with inexact-map perturbation, the next
proposition keeps the iterates in the enlarged sublevel set and gives an affine
contraction recursion for sufficiently small $\epsilon$.
\begin{proposition}\label{prop: inexact_outer_control}
    Define
    \begin{align}
        r_0 := & \operatorname{dist}(\K(\mathcal{L}_{\leq v(x_0)+1}),\partial P)>0, \\
        \mathcal{V}_0 := & \left\{y\in \R^n:\operatorname{dist}(y,\K(\mathcal{L}_{\leq v(x_0)+1}))\le \frac{r_0}{2}\right\}, \\
        \label{eq: inexact_L0}
        L_0 := & \max_{y\in \mathcal{V}_0}\|\nabla \delta_v(y)\|,\\
        C_0 := & L_0B_0.
    \end{align}
    Then the following claims hold.
    \begin{enumerate}
        \item The set $\mathcal{V}_0$ is compact and contained in $\operatorname{int}(P)$, and $0< L_0<+\infty$. 
        \item If
        \[
            \epsilon\le \frac{\lambda_0^2\rho_0^2}{2m^2}
            \qquad \text{and} \qquad
            B_0\sqrt{\epsilon}\le \frac{r_0}{2},
        \]
        then for every $x\in \mathcal{L}_{\leq v(x_0)+1}$,
        \[
            \widehat{\K}(x)\in \mathcal{V}_0
            \qquad \text{and} \qquad
            \delta_v(\widehat{\K}(x))
            \le
            q_0\delta_v(x)+C_0\sqrt{\epsilon}.
        \]
        \item If, in addition,
        \[
            C_0\sqrt{\epsilon}\le 1-q_0,
        \]
        then the iterates generated by Algorithm \ref{alg: inexact_polarity_process}
        satisfy, for all $k\geq0$,
        \[
            x_k\in \mathcal{L}_{\leq v(x_0)+1},
        \]
        and
        \[
            \delta_v(x_k)
            \le
            q_0^k\delta_v(x_0)+\frac{C_0}{1-q_0}\sqrt{\epsilon}.
        \]
    \end{enumerate}
\end{proposition}
\begin{proof}
    \begin{enumerate}
        \item Lemmas \ref{lemma: lifted_MinCE_continuity} and
        \ref{lemma: inexact_sublevel_set} make
        $\K(\mathcal{L}_{\leq v(x_0)+1})$ compact. Lemma
        \ref{lemma: inexact_basic_properties}, applied to the exact optimizer
        $\mathcal Q(x)$, gives a full-dimensional ellipsoid centered at
        $\K(x)$ and contained in $P$; hence $\K(x)\in\operatorname{int}(P)$.
        Therefore the image is a compact subset of $\operatorname{int}(P)$,
        so $r_0>0$ and $\mathcal V_0$ is compact in
        $\operatorname{int}(P)$. Proposition
        \ref{proposition: v_gradient} and Lemma
        \ref{lemma: lifted_MinCE_continuity} make
        $\nabla\delta_v=(n+1)c$ continuous on $\mathcal V_0$, so
        $L_0<\infty$.

        Lemma \ref{lemma: unique_minimizer_v} places $x^*$ in
        $\mathcal{L}_{\leq v(x_0)+1}$, while Lemma
        \ref{lemma: facts_about_g} and
        \eqref{eq: self_contained_K_formula} give $\K(x^*)=x^*$.
        Thus $\mathcal V_0$ contains a nontrivial ball about $x^*$. If
        $y\ne x^*$ and $c(y)=\mathbf0$, Proposition
        \ref{proposition: v_gradient} and convexity make $y$ a global
        minimizer of $v$, contradicting Lemma
        \ref{lemma: unique_minimizer_v}. Therefore
        $\nabla\delta_v(y)=(n+1)c(y)\ne\mathbf0$ there, and $L_0>0$.

        \item Let $x\in \mathcal{L}_{\leq v(x_0)+1}$. By Lemma \ref{lemma: inexact_oracle_recovery} we have
        \[
            \|\widehat{\K}(x)-\K(x)\|\le B_0\sqrt{\epsilon}\le \frac{r_0}{2}.
        \]
        Every point on the segment joining $\K(x)$ and $\widehat\K(x)$ is
        within $r_0/2$ of $\K(x)\in\K(\mathcal L_{\leq v(x_0)+1})$, so the
        entire segment lies in $\mathcal V_0$. The definition of $L_0$ gives
        \[
            \delta_v(\widehat{\K}(x))
            \le
            \delta_v(\K(x))+L_0\|\widehat{\K}(x)-\K(x)\|.
        \]
        Lemmas \ref{lemma: inexact_sublevel_set} and
        \ref{lemma: inexact_oracle_recovery} then yield
        \[
            \delta_v(\widehat{\K}(x))
            \le
            q_0\delta_v(x)+L_0B_0\sqrt{\epsilon}
            =
            q_0\delta_v(x)+C_0\sqrt{\epsilon}.
        \]

        \item Let $x_0\in\mathcal{L}_{\leq v(x_0)+1}$, part (2)
        applied inductively for $0\leq k<N$ gives
        \begin{align*}
            \delta_v(x_{k+1}) &\le q_0\delta_v(x_k)+C_0\sqrt{\epsilon} \\
            &\le q_0(\delta_v(x_0)+1)+(1-q_0) \\
            &=q_0\delta_v(x_0)+1
            \le\delta_v(x_0)+1.
        \end{align*}
        so every produced iterate remains in the enlarged sublevel set. Part
        2 then applies for $0\leq k<N$ where $N\in \N$: 
        \[
            \delta_v(x_{k+1})
            \le
            q_0\delta_v(x_k)+C_0\sqrt{\epsilon}, \qquad 0\leq k<N.
        \]
        Iterating gives
        \[
            \delta_v(x_k)
            \le
            q_0^k\delta_v(x_0)+\frac{C_0}{1-q_0}\sqrt{\epsilon},
            \qquad 0\leq k\leq N.
        \]
    \end{enumerate}
    This completes the proof. 
\end{proof}

The next theorem uses the preceding invariance and error bounds to choose a fixed
tolerance $\epsilon$ and iteration bound sufficient for approximation factor
$\gamma$.
\begin{theorem}\label{thm: inexact_fixed_tolerance}
    Choose $x_0\in \mathrm{int}(P)$,  fix $\gamma\in (0,1)$, and let
    \[
        \epsilon
        :=
        \min\left\{
            \frac{\lambda_0^2\rho_0^2}{2m^2},
            \frac{\lambda_0^2r_0^2}{8m^2L_{\mathrm{ctr}}^2},
            \frac{\lambda_0^2(1-q_0)^2}{2m^2L_0^2L_{\mathrm{ctr}}^2},
            \frac{\lambda_0^2(1-q_0)^2 \log(1/\gamma)^2}{32m^2L_0^2L_{\mathrm{ctr}}^2},
            \frac{\lambda_0^2\log(1/\gamma)^2}{32m^2L_{\mathrm{vol}}^2}
        \right\}.
    \]
    Here $q_0$, $\rho_0$, and $\lambda_0$ are as in \eqref{eq: inexact_q0}, \eqref{eq: inexact_rho0}, and \eqref{eq: inexact_lambda0}, respectively;
    $L_{\mathrm{ctr}}$ and $L_{\mathrm{vol}}$ are as in \eqref{eq: inexact_lipschitz_constants_center} and \eqref{eq: inexact_lipschitz_constants_vol}; and $L_0$ is as in \eqref{eq: inexact_L0}.
    Then Algorithm \ref{alg: inexact_polarity_process} is well-defined, and the ellipsoid
    $\E(x_{N_\gamma+1},E_{N_\gamma+1})$ is $\gamma$-maximal with
    \[
        N_\gamma
        :=
        \left\lceil
            \frac{\log^+(2\delta_v(x_0)/\log(1/\gamma))}{\log(1/q_0)}
        \right\rceil.
    \]
\end{theorem}
\begin{proof}
    Let $L_\gamma:=\log(1/\gamma)$. The first bound on $\epsilon$ gives $\epsilon \le \frac{\lambda_0^2\rho_0^2}{2m^2}.$
    The second and third bounds, with $B_0$ and $C_0$ in Lemma \ref{lemma: inexact_oracle_recovery} and Proposition \ref{prop: inexact_outer_control}, give
    \[
        B_0\sqrt{\epsilon}\le \frac{r_0}{2},
        \quad
        C_0\sqrt{\epsilon}\le 1-q_0.
    \]
    The fourth and fifth bounds, with $D_0$ in Lemma \ref{lemma: inexact_oracle_recovery}, similarly give
    \[
        \frac{C_0}{1-q_0}\sqrt{\epsilon}\le \frac{L_\gamma}{4},
        \quad
        D_0\sqrt{\epsilon}\le \frac{L_\gamma}{4}.
    \]
    Lemma \ref{lemma: inexact_oracle_recovery}  and Proposition \ref{prop: inexact_outer_control} make the algorithm
    well-defined, and give
    \[
        x_k\in \mathcal{L}_{\leq v(x_0)+1},
        \quad
        \delta_v(x_k)
        \le
        q_0^k\delta_v(x_0)+\frac{L_\gamma}{4},
        \quad 0\leq k\leq N_\gamma+1.
    \]

    For each $0\leq k\leq N_\gamma$, the exact primal update at $x_k$ satisfies
    \[
        \log \frac{\V(\MaxIE(P))}{\V(\E(\K(x_k),E(x_k)))}
        =
        \delta_v(x_k)-\frac{n+1}{2}\log \frac{1}{1-\tau(x_k)}
        \le
        \delta_v(x_k),
    \]
    where Lemmas \ref{lemma: useful_lemma} and
    \ref{lemma: facts_about_g}, together with the definition of
    $\delta_v$, are used.
    Lemma \ref{lemma: inexact_oracle_recovery} then yields
    \begin{align*}
        &\log \frac{\V(\MaxIE(P))}{\V(\E(x_{k+1},E_{k+1}))} \\
        = & \log \frac{\V(\MaxIE(P))}{\V(\E(\K(x_k),E(x_k)))}  + \log \frac{\V(\E(\K(x_k),E(x_k)))}{\V(\E(x_{k+1},E_{k+1}))} \\
        \quad\le & \delta_v(x_k) +
        \left|
            \log\V(\E(x_{k+1},E_{k+1}))
            -
            \log\V(\E(\K(x_k),E(x_k)))
        \right| \\
        \le & \delta_v(x_k)+D_0\sqrt{\epsilon} \le q_0^k\delta_v(x_0)+\frac{L_\gamma}{2}.
    \end{align*}
    At $k=N_\gamma$, its definition gives $ q_0^{N_\gamma}\delta_v(x_0)\le \frac{L_\gamma}{2}$, and hence
    \[
        \log \frac{\V(\MaxIE(P))}{\V(\E(x_{N_\gamma+1},E_{N_\gamma+1}))}
        \le
        L_\gamma
        =
        \log(1/\gamma);
    \]
    equivalently, we have 
    \[
        \V(\E(x_{N_\gamma+1},E_{N_\gamma+1}))
        \ge
        \gamma\,\V(\MaxIE(P)).
    \]
    This completes the proof.
\end{proof}
We acknowledge that the sufficient oracle tolerance is nonconstructive because its constants depend
on compact initial-sublevel geometry.

\subsection{Arithmetic counts based on inexact MinCE oracles}\label{subsec: arith}
We use existing MinCE solvers to derive conditional arithmetic estimates.
\subsubsection{Barycentric coordinate descent and path-following Newton method}\label{subsec: arith_khachiyan}
Khachiyan \cite{khachiyan1996rounding} gave complexities for two
inexact MinCE oracles. We first translate his notion of inexactness into our Definition \ref{def: inexact oracle}.

\begin{lemma}\label{lemma: lowner_to_obj_gap}
    Given $\epsilon >0$, let
    \begin{align}\label{eq: eta}
        \eta := \min \left \{1, \exp\left(\frac{\epsilon}{2}\right) - 1 \right \}.
    \end{align}
    Let $x \in \mathrm{int}(P)$. If
    \begin{align}\label{eq: lowner_problem}
        \Vol (\E(\mathbf{0}, \wh{\mathcal{Q}}(x)) )\leq (1  + \eta)  \Vol (\E(\mathbf{0}, \mathcal{Q}(x))),
    \end{align}
    then
     \[
        -\log \det \wh{\mathcal{Q}}(x)  \leq -\log \det \mathcal{Q}(x) + \epsilon.
    \]
\end{lemma}
\begin{proof}
Equations \eqref{eq: lowner_problem} and \eqref{eq: eta} give
\begin{align*}
    -\log \det \wh{\mathcal{Q}}(x) \leq -\log \det \mathcal{Q}(x) + 2\log(1 +\eta)\leq -\log \det \mathcal{Q}(x) + \epsilon.
\end{align*}
\end{proof}
The condition $\eta\leq1$ in \eqref{eq: eta} makes Khachiyan's results in
\cite{khachiyan1996rounding} applicable, so any feasible solution of \eqref{eq: lifted_MinCE} satisfying
\eqref{eq: lowner_problem} is an adequate oracle. He proposed solving this
L\"owner problem through an equivalent determinant-maximization problem over
$\Delta_{2m}$. Let
\[
    \mathcal{P}(x) = \{p_1(x),\cdots, p_{2m}(x)\}
\]
be an enumeration of the set
\[
    \left\{\pm \begin{bmatrix} \tilde{a}_1(x) \\1 \end{bmatrix}, \cdots, \pm \begin{bmatrix} \tilde{a}_m(x) \\1 \end{bmatrix} \right\}.
\]
Then \eqref{eq: dual_MinCE} is equivalent to
\begin{align}\label{eq: equivalent_dual_MinCE}
    \max_{y \in \Delta_{2m}} \quad \log \det  \sum_{j=1}^{2m} y_j p_j(x) p_j(x)^\top.
\end{align}
Khachiyan gives both barycentric coordinate descent (BCD), with primal recovery
satisfying \eqref{eq: lowner_problem}, and a path-following Newton bound
\cite{nesterov1989self,nesterov1994interior}.

\begin{theorem}{\cite[Theorems 3 and 4]{khachiyan1996rounding}}\label{thm: khachiyan}
    A feasible solution $\wh{\mathcal{Q}}(x) \in \S^{n+1}_{++}$ of \eqref{eq: lifted_MinCE} satisfying \eqref{eq: lowner_problem} can be obtained by solving the equivalent dual MinCE
    problem \eqref{eq: equivalent_dual_MinCE} with either
    \begin{align}
        \mathcal{O}\left(mn^2 \left(\frac{n}{\eta} + \log \log m \right) \right)
    \end{align}
    arithmetic operations by the barycentric coordinate descent (BCD) method, or
     \begin{align}
        \mathcal{O}\left( m^{3.5} \log \left(\frac{m}{\eta}\right)\right)
    \end{align}
    arithmetic operations using pre-rounding and path-following Newton method.
\end{theorem}

Theorems \ref{thm: inexact_fixed_tolerance} and \ref{thm: khachiyan} yield
bounds with instance-dependent contraction, compactness, and Lipschitz
constants not uniform in $m,n,$ and $R$.
\begin{theorem}\label{thm: overall_arithmetic_complexity}
    Let $x_0\in \mathrm{int}(P)$ and $\gamma\in (0,1)$. Let $\epsilon$ and $N_\gamma$ be as in Theorem \ref{thm: inexact_fixed_tolerance}, and let $\eta$ be as in \eqref{eq: eta}.
    If this instance-dependent tolerance and budget are supplied, the polarity
    process produces a $\gamma$-maximal ellipsoid of $P$ using
    \begin{align*}
        \mathcal{T}_{\text{PP-BCD}} =
        \mathcal{O}\left(
             (N_\gamma+1)\, mn^2
            \left(
                \frac{n}{\eta} + \log\log m
            \right)
        \right)
    \end{align*}
    arithmetic operations when BCD is used to solve each dual MinCE problem \eqref{eq: equivalent_dual_MinCE}, or at most
    \begin{align*}
        \mathcal{T}_{\text{PP-IPM}} =  \mathcal{O}\left(
            (N_\gamma+1)\, m^{3.5}\log\left(\frac{m}{\eta}\right)
        \right)
        \label{eq: overall_arithmetic_complexity_newton}
    \end{align*}
    arithmetic operations when each dual MinCE problem \eqref{eq: equivalent_dual_MinCE} is solved by a pre-rounding procedure followed by the path-following Newton method.
\end{theorem}
\begin{proof}
    The proof follows immediately from Theorems \ref{thm: inexact_fixed_tolerance} and \ref{thm: khachiyan}.
\end{proof}

To compare our overall bounds with the direct MaxIE bound of Khachiyan and Todd
\cite{khachiyan1990complexity}, write $L_\gamma:=\log(1/\gamma)$.
For the following simplified comparison, take $x_0=\mathbf0$ under the centered
ball sandwich and hold the outer geometric constants fixed: $q_0$ stays away
from one; $\lambda_0$, $\rho_0$, and $r_0$ stay away from zero; and $L_0$,
$L_{\mathrm{ctr}}$, and $L_{\mathrm{vol}}$ stay bounded. Then, as
$L_\gamma\downarrow0$, the tolerance in Theorem
\ref{thm: inexact_fixed_tolerance} satisfies $ \epsilon=\Theta(L_\gamma^2/m^2).$
Since $e^t=1+t+\mathcal O(t^2)$ for small $t$, equation~\eqref{eq: eta} gives $\eta = \Theta(L_\gamma^2/m^2)$
as well. 
Invoking Corollary
\ref{corollary: bounded radius ratio} and Theorem
\ref{thm: inexact_fixed_tolerance}, we see
\[
    N_\gamma+1
    =\mathcal O\!\left(1+\log^+\!\frac{2n\log R}{L_\gamma}\right),
\]
where $R\geq1$ is the ratio of the radii of the outer and inner Euclidean balls.
Substituting these high-accuracy quantities into
Theorem~\ref{thm: overall_arithmetic_complexity} gives
\[
    \mathcal{T}_{\text{PP-BCD}}
    =\mathcal O\!\left(
      \left(\frac{m^3n^3}{L_\gamma^2}+mn^2\log\log m\right)
      \left[1+\log^+\!\frac{2n\log R}{L_\gamma}\right]
    \right)
\]
and
\[
    \mathcal{T}_{\text{PP-IPM}}
    =\mathcal O\!\left(
      m^{3.5}\log\!\left(\frac{m^3}{L_\gamma^2}\right)
      \left[1+\log^+\!\frac{2n\log R}{L_\gamma}\right]
    \right).
\]
Here we assume $\gamma$ is big enough so that the value of $\log(\cdot)$ is positive. 

Khachiyan and Todd \cite{khachiyan1990complexity} gave this uniform worst-case
bound for \texttt{Ellip}:
\[
    \mathcal O\!\left(
      m^{3.5}
      \log\left(\frac{mR}{L_\gamma}\right)
      \left[1+\log^+\!\frac{n\log R}{L_\gamma}\right]
    \right).
\]
Anstreicher \cite{anstreicher2002improved} later obtained the direct MaxIE bound
\[
    \mathcal O\!\left(
      m^{3.5}\log\left(\frac{mR}{L_\gamma}\right)
    \right),
\]
where his $e^{-\varepsilon}$-accuracy notation is translated by setting
$\varepsilon=L_\gamma$. Thus Anstreicher removes the additional logarithmic
factor appearing in the Khachiyan--Todd bound. Zhang and Gao
\cite{zhanggao2003mvie} take an implementation-oriented primal--dual
interior-point approach: each dense Newton step costs $\mathcal O(m^3)$ because
it uses dense $m$-dimensional linear algebra. They compare per-iteration cost
and observed iteration counts, not a directly comparable global
$\gamma$-maximal bound.

With fixed geometric constants, the Newton-based polarity-process and direct
interior-point estimates share the $m^{3.5}$ factor, but the former retains the
outer-iteration factor. Relative to \texttt{Ellip}, it differs mainly in the
first logarithmic factor. More generally, suppose
$L_\gamma\leq m^{1-\alpha}$ for some $\alpha>0$.
Then
\[
    \log\!\left(\frac{m^3}{L_\gamma^2}\right)
    \leq\left(2+\frac1\alpha\right)\log\!\left(\frac{m}{L_\gamma}\right)
    =\mathcal O\!\left(\log\!\frac{m}{L_\gamma}\right).
\]
For $\gamma\geq1/e$, this holds with $\alpha=1$, including the high-accuracy
regime $\gamma\uparrow1$. In this regime, the first logarithmic factor in the
polarity-process estimate has no explicit $R$-dependence when the outer
geometric constants are fixed, thanks to Khachiyan's pre-rounding procedure
\cite{khachiyan1996rounding}. Our $\mathcal{T}_{\text{PP-IPM}}$ bound does not improve
Anstreicher's direct worst-case bound. This structurally compares an iterative
geometric process with a direct interior-point method; it does not claim
worst-case dominance.

Sun and Freund \cite{sunfreund2004mvce} developed the dual reduced Newton (DRN)
algorithm, whose Newton systems exploit the MinCE structure more sharply than
generic determinant-maximization interior-point machinery. They also show that
the DRN objective is not self-concordant but is closely related to a self-concordant function. Thus, while DRN is
computationally effective, especially with active-set strategies, it does not
plug as directly as the self-concordant path-following framework into an overall
arithmetic bound for the outer polarity process.

The BCD estimate has a lower power of $m$ but stronger dependence on $n$ and
target accuracy. Neither estimate dominates uniformly, reflecting the cost of a
first-order inner method. We next examine WA-TY.

\subsubsection{The WA-TY modified Frank--Wolfe method}\label{subsec: arith_waty}
Ahipasaoglu et al. \cite{ahipasaoglu2008fw} viewed Khachiyan's BCD as
conditional gradient (Frank--Wolfe) on the dual simplex. They established
local linear convergence for its Wolfe--Atwood--Todd--Yildirim (WA-TY)
modification with away and drop steps. We translate this fixed-instance dual
MinCE result into an inexact lifted oracle and then into the double-looped
polarity process.

Fix $x \in \mathrm{int}(P)$ and write $d:=n+1$. Since for every $\lambda\in \Delta_m$,
\[
    \sum_{i=1}^m \lambda_i
    \begin{bmatrix}
        \tilde a_i(x)\\1
    \end{bmatrix}
    \begin{bmatrix}
        \tilde a_i(x)\\1
    \end{bmatrix}^{\top}
    =
    \begin{bmatrix}
        A(x,\lambda) & c(x,\lambda) \\
        c(x,\lambda)^\top & 1
    \end{bmatrix},
\]
using \eqref{eq: dual_MinCE_quantities} and $\sum_i\lambda_i=1$.
Thus \eqref{eq: dual_MinCE} is exactly the fixed-instance simplex
determinant-maximization problem considered in \cite{ahipasaoglu2008fw}. For $\lambda\in \Delta_m$, define
\[
    \mathsf M(x,\lambda):=
    \begin{bmatrix}
        A(x,\lambda) & c(x,\lambda)\\
        c(x,\lambda)^\top & 1
    \end{bmatrix}.
\]
When $\mathsf M(x,\lambda)\succ0$, let
\[
    w_i(x,\lambda):=
    \begin{bmatrix}
        \tilde a_i(x)\\1
    \end{bmatrix}^{\top}
    \mathsf M(x,\lambda)^{-1}
    \begin{bmatrix}
        \tilde a_i(x)\\1
    \end{bmatrix},
    \qquad i\in[m].
\]
Following \cite{ahipasaoglu2008fw}, we say
$\lambda\in\Delta_m$ with $\mathsf M(x,\lambda)\succ0$ is a
\emph{strong $\zeta$-approximate maximizer} of \eqref{eq: dual_MinCE} if
\begin{subequations}
\begin{align}
    w_i(x,\lambda)\leq & (1+\zeta)d,\quad \forall i\in [m], \label{eq: strong_zeta_upper}\\
    w_i(x,\lambda)\geq & (1-\zeta)d,\quad \forall i\in \mathrm{supp}(\lambda). \label{eq: strong_zeta_lower}
\end{align}
\end{subequations}
The upper inequality \eqref{eq: strong_zeta_upper} implies that
\begin{align}\label{eq: waty_recovery}
    \wh{\mathcal{Q}}(x;\lambda,\zeta):=
    \frac{1}{d(1+\zeta)}
    \mathsf M(x,\lambda)^{-1}
\end{align}
is feasible for the lifted MinCE problem \eqref{eq: lifted_MinCE}. Indeed, for every $i$,
\[
    \begin{bmatrix}
        \tilde a_i(x)\\1
    \end{bmatrix}^{\top}
    \wh{\mathcal{Q}}(x;\lambda,\zeta)
    \begin{bmatrix}
        \tilde a_i(x)\\1
    \end{bmatrix}
    = \frac{w_i(x,\lambda)}{d(1+\zeta)}
    \leq 1.
\]
Thus such a point yields a feasible lifted matrix for the outer process.

\begin{theorem}{\cite[Proposition~2.5 and Theorem~2.6(a)]{ahipasaoglu2008fw}}\label{thm: waty_fixed_instance}
    Fix $x\in \mathrm{int}(P)$, consider \eqref{eq: dual_MinCE}, and initialize WA-TY with $\lambda_i^0=1/m$. Adapting the cited iteration bound, there are
    positive, data-dependent constants $C_{\mathrm{WA}}(x)$ and $J(x)$ such that, for every
    $\zeta\in(0,1)$, WA-TY requires at most
    \begin{align}\label{eq: waty_inner_iterations}
        J(x)+56C_{\mathrm{WA}}(x)\log\left(\frac{1}{\zeta}\right)
    \end{align}
    iterations to produce a strong $\zeta$-approximate maximizer
    $\lambda_\zeta(x)$. The recovered matrix from
    \eqref{eq: waty_recovery} satisfies the stronger primal estimate
    \begin{align}\label{eq: waty_primal_gap}
        0
        \leq
        -\log\det \wh{\mathcal{Q}}(x;\lambda_\zeta(x),\zeta)
        +\log\det \mathcal{Q}(x)
        \leq d\log(1+\zeta).
    \end{align}
\end{theorem}
\begin{proof}
    The cited perturbation estimate first gives
    \eqref{eq: waty_inner_iterations} for all sufficiently small $\zeta$, say
    $0<\zeta\leq\zeta_0(x)<1$. For a requested
    $\zeta\in(\zeta_0(x),1)$, it is enough to run to the stronger
    $\zeta_0(x)$-accuracy; enlarging the data-dependent transient constant by
    $56C_{\mathrm{WA}}(x)\log(1/\zeta_0(x))$ gives the displayed bound for every
    $\zeta\in(0,1)$. Thus no knowledge of $\zeta_0(x)$ is needed in the outer
    tolerance choice. For the full-support initialization used here, $J(x)$ also absorbs the finite initial phase and the at most $m$ drop iterations associated with initially supported components; every other drop iteration is paired with an earlier add iteration in the cited proof. The constant $C_{\mathrm{WA}}(x)$ comes from Proposition~2.5 of the
    cited source. Equation \eqref{eq: strong_zeta_upper} makes the recovered
    matrix feasible. Let $g^*(x):=\max_{\mu\in\Delta_m}\log\det\mathsf M(x,\mu).$
    Direct calculation gives
    \begin{align*}
        &-\log\det \wh{\mathcal Q}(x;\lambda,\zeta)
        +\log\det\mathcal Q(x)\\
        &\qquad=d\log(1+\zeta)
        -\bigl[g^*(x)-\log\det\mathsf M(x,\lambda)\bigr] \leq d\log(1+\zeta).
    \end{align*}
    The lower bound in \eqref{eq: waty_primal_gap} follows from optimality of
    $\mathcal Q(x)$.
\end{proof}

For arithmetic conversion, use the full-support initialization $\lambda_i=1/m$.
The lifted points
span $\R^d$, so $\mathsf M(x,\lambda)\succ0$; forming this matrix and its inverse
costs $\mathcal O(md^2+d^3)=\mathcal O(md^2)$ because $m\geq d$. A dense
implementation evaluating all $m$ leverage scores and updating the inverse uses
$\mathcal O(md^2)$ work per iteration. By \eqref{eq: waty_inner_iterations},
WA-TY's arithmetic complexity is
\begin{equation}\label{eq: WA-TY-complexity}
    \mathcal O\!\left(
    md^2\left[J(x)+56C_{\mathrm{WA}}(x)\log(1/\zeta)\right]\right).
\end{equation}

\begin{theorem}[Realized-trajectory WA-TY accounting]\label{thm: overall_arithmetic_complexity_waty}
    Let $x_0\in\mathrm{int}(P)$ and $\gamma\in(0,1)$, and let $\epsilon$ and
    $N_\gamma$ be as in Theorem~\ref{thm: inexact_fixed_tolerance}. Fix
    \begin{align}\label{eq: zeta_gamma_condition}
        \zeta_\gamma:=\min\left\{\frac12,\frac{\epsilon}{n+1}\right\}.
    \end{align}
    Run Algorithm~\ref{alg: inexact_polarity_process} for $N_\gamma+1$ updates,
    using WA-TY to strong $\zeta_\gamma$-accuracy at each realized outer point.
    Afterward define
    \begin{align}\label{eq: waty_const}
        C_{\mathrm{WA},\gamma}:=\max_{0\leq k\leq N_\gamma}C_{\mathrm{WA}}(x_k),
        \qquad
        J_\gamma:=\max_{0\leq k\leq N_\gamma}J(x_k).
    \end{align}
    The last ellipsoid is $\gamma$-maximal, and the dense implementation model
    above gives the realized-trajectory upper bound
    \begin{align*}
        \mathcal T_{\mathrm{PP-WA-TY}}
        =\mathcal O\!\left(
          (N_\gamma+1)m(n+1)^2
          \left[J_\gamma+56C_{\mathrm{WA},\gamma}
          \log\left(\frac1{\zeta_\gamma}\right)\right]
        \right).
    \end{align*}
\end{theorem}
\begin{proof}
    With $d=n+1$, \eqref{eq: zeta_gamma_condition} gives
    $d\log(1+\zeta_\gamma)\leq\epsilon$: if
    $\zeta_\gamma=\epsilon/d$, use $\log(1+t)\leq t$; otherwise
    $\epsilon\geq d/2>d\log(3/2)$. Thus Theorem
    \ref{thm: waty_fixed_instance} supplies a feasible $\epsilon$-accurate
    lifted oracle at every realized iterate. Apply Theorem
    \ref{thm: inexact_fixed_tolerance} and sum the fixed-instance iteration
    counts using \eqref{eq: WA-TY-complexity} and \eqref{eq: waty_const}.
\end{proof}

Under the same fixed-geometric-constants interpretation and as $\gamma\uparrow1$,
$\zeta_\gamma=\Theta(L_\gamma^2/[m^2(n+1)])$. The corresponding displayed
cost may be written
\[
    \mathcal O\!\left(
      m(n+1)^2\left[
      J_\gamma+C_{\mathrm{WA},\gamma}\log\!\left(
      \frac{m^2(n+1)}{L_\gamma^2}\right)\right]
      \left[1+\log^+\!\frac{2n\log R}{L_\gamma}\right]
    \right).
\]
From the leading factors, $\mathcal T_{\mathrm{PP-WA-TY}}$ has
$m(n+1)^2$, compared with $m^3n^3$ in $\mathcal T_{\mathrm{PP-BCD}}$ and
$m^{3.5}$ in $\mathcal T_{\mathrm{PP-IPM}}$ and \cite{anstreicher2002improved,khachiyan1990complexity}.
Because the lifted points span $\R^{n+1}$, $n+1\leq m$, hence
$m(n+1)^2=\mathcal O(m^3)$. WA--TY also replaces BCD's inverse-quadratic
dependence on $L_\gamma$ by a logarithmic term. Thus, when $J_\gamma$ and
$C_{\mathrm{WA},\gamma}$ are moderate, the displayed WA--TY expression has the
smallest displayed explicit polynomial prefactor in $m$ and $n$ among these bounds.
This remains a posteriori work accounting for the realized trajectory, not an \emph{a
priori} family-uniform theorem or a claim of uniform worst-case dominance;
such a theorem would require known uniform bounds on $C_{\mathrm{WA}}(x)$ and
$J(x)$ over the shifted polar MinCE family.

\section{Numerical experiments}\label{sec:numerical-experiments}
We compare practical inner oracles for the subproblem $\mathrm{MinCE}\bigl((P-x_k)^\circ\bigr)$
within the polarity process against the Khachiyan--Todd \texttt{Ellip}
framework \cite{khachiyan1990complexity} and Zhang and Gao's F2PD 
\cite{zhanggao2003mvie}. Our codes are implemented in Julia \cite{bezanson2017julia}
and executed with 8 threads on a 10-core Apple M4
with 32 GB of memory.

\paragraph{Test instances.}
We test four families of \eqref{eq: polytope}, whose
$\mathrm{A}\in\R^{m\times n}$ rows are $a_i^\top$:
\begin{enumerate}
\item \textbf{Box.} $\mathrm{A}=[I_n;-I_n]$, so $P=[-1,1]^n$ and $m=2n$.
\item \textbf{Skewed box.} $\mathrm{A}=[B;-B]$, where
$B=U\Diag(s)V^\top$, where $U$ and $V$ are the left and right singular-vector
matrices of a Gaussian random matrix and $s$ is evenly spaced in
$[0.45,1.85]$; again $m=2n$.
\item \textbf{Random facets.} We form $\mathrm{A}$ by augmenting $[I_n;-I_n]$ with
$m_{\mathrm{extra}}$ normalized Gaussian normals, each multiplied by an
independent uniform factor in $[0.55,1.40]$, giving
$m=2n+m_{\mathrm{extra}}$.
\item \textbf{Cross-polytope.} All vectors in $\{\pm1\}^n$ are facet normals, so
$m=2^n$.
\end{enumerate}
We fix all random-construction and start-direction seeds. Thus each displayed
$(n,m)$ uses one deterministic instance and one start-direction pair. These are
exploratory case studies, not a multi-seed statistical performance profile.


For each instance we use three interior starts, $x^{(c)}$, $x^{(m)}$, and
$x^{(b)}$. The central point $x^{(c)}$ solves
\[
    \max_{x\in \R^n,\ t\ge 0}\ \{t: a_i^\top x \le 1-t,\ i\in [m]\}.
\]
We normalize independent Gaussian directions $d_m,d_b$ and define
\[
    \lambda_{\max}(d)
    :=
    \sup\{\lambda\ge 0: x^{(c)}+\lambda d \in P\}.
\]
We then set
\[
    x^{(m)} = x^{(c)} + 0.45\,\lambda_{\max}(d_m)d_m,
    \qquad
    x^{(b)} = x^{(c)} + 0.90\,\lambda_{\max}(d_b)d_b.
\]
These are the central (c), midway (m), and near-boundary (b) initial points used
by every method.
For the centrally symmetric box, skewed-box, and cross-polytope families,
$x^{(c)}$ is also the MaxIE center, so the central rows are sanity checks and
remain included in the reported Tables \ref{tab:benchmark_box_times}-\ref{tab:benchmark_crosspoly_times}.

\paragraph{Compared methods.}
Table~\ref{tab:benchmark_methods} lists two direct MaxIE baselines and five PP
variants.
\begin{table}[!htbp]
\centering
\small
\begin{tabular}{@{}>{\raggedright\arraybackslash}p{0.27\textwidth}@{\hspace{1em}}>{\raggedright\arraybackslash}p{0.69\textwidth}@{}}
\toprule
Abbreviation & Description \\
\midrule
\texttt{Ellip} & The direct MaxIE framework of Khachiyan
and Todd \cite{khachiyan1990complexity}, with its subproblems solved by the
general-purpose nonlinear interior-point solver \texttt{IPOPT}
\cite{wachter2006implementation}. \\ \hline
\texttt{F2PD} & Our Julia adaptation of the Zhang--Gao direct primal--dual
iteration \cite{zhanggao2003mvie} distributed in the MATLAB package \cite{cousins2022volumeandsampling}, using
Julia linear algebra and regularization. \\ \hline
\texttt{PP-IPOPT} & Polarity process (PP) with the \texttt{IPOPT} applied to the dual MinCE reformulation
\eqref{eq: equivalent_dual_MinCE}. \\ \hline
\texttt{PP-DRN-AS} & PP with our implementation of the Sun--Freund dual
reduced Newton (DRN) method \cite{sunfreund2004mvce}, embedded in the code's
active-set (AS) heuristic. \\ \hline
\texttt{PP-BCD} & PP with our implementation of Khachiyan's barycentric
coordinate descent (BCD) method \cite{khachiyan1996rounding}. \\ \hline
\texttt{PP-WA-TY-Kha} & PP with the Wolfe--Atwood--Todd--Yildirim (WA-TY)
away-step method and a Khachiyan-style uniform (Kha) initialization
\cite{ahipasaoglu2008fw}. \\ \hline
\texttt{PP-WA-TY-KY} & PP with WA-TY and Kumar--Y{\i}ld{\i}r{\i}m
(KY) sparse initialization \cite{ahipasaoglu2008fw,kumar2005minimum}. \\
\bottomrule
\end{tabular}
\caption{Benchmark abbreviations and their algorithmic origins.}
\label{tab:benchmark_methods}
\end{table}

We omit Anstreicher's improved direct MaxIE method
\cite{anstreicher2002improved} as a baseline because the paper gives an
important worst-case complexity refinement but no readily reproducible
implementation. We also omit recent leverage-score algorithms as inner-oracle
baselines. Zhao's away-step method \cite{Zhao2025AwayStep} belongs to the family represented by
\texttt{PP-WA-TY}. The others concern fixed centered
instances, while each polarity step changes the lifted rows. In particular, work
\cite{LiYuJiangGaoHan2026} requires exact scores, nondegeneracy, and
condition-dependent face identification not controlled uniformly along this
trajectory. The lazy/streaming, sketched input-sparsity, and small-treewidth
models of \cite{WoodruffYasuda2024,CaoLiSongYangZhou2025} also differ from our
solver-timing setting.


All polarity variants and \texttt{Ellip} use the following outer rule for
$\mathcal{E}_k=\E(x_k,H_k)$ and its normalized volume
$V_k=\det(H_k)^{-1/2}$:
\[
    \frac{|V_k-V_{k-1}|}{\max\{V_k,V_{k-1}\}} \le 10^{-8}.
\]
This is a practical volume-stabilization stopping heuristic. Runs are capped
at $10,000$ outer iterations and use a nominal wall-clock budget
$T=1800$ seconds.
The base inner-solver target is $\tau_{\rm sol}=10^{-9}$, interpreted according
to each implementation as we explain next. \texttt{Ellip} calls \texttt{IPOPT} with this tolerance
and a $100,000$-iteration cap per subproblem. \texttt{PP-IPOPT} uses the same
base settings, while its robustness retries can relax the tolerance to at most
$20\tau_{\rm sol}$ and the per-attempt cap to $200,000$. BCD and WA-TY allow
at most $100,000$ weight updates per MinCE call and target $ \max\{\kappa_+-d,d-\kappa_-\}\le \tau_{\rm sol},$
where $d=n+1$, $\kappa_+$ is the largest lifted leverage score, and
$\kappa_-$ is the smallest score on the current support. For
\texttt{PP-DRN-AS}, every reduced Newton solve has a $100,000$-iteration cap
and targets
\[
    \|{\bf 1}-t-h(u)\|_\infty \le \tau_{\rm sol}
    \quad\text{and}\quad
    \frac{u^\top t}{\max\{|f(u)|,1\}} \le \tau_{\rm sol},
\]
where $u>0$ and $t$ are the dual-weight and slack vectors, $h_i(u)$ are leverage
scores, and $f(u)=-\tfrac12\log\det Q(u)$ for the current DRN shape matrix $Q(u)$. The
active-set driver allows at most $50$ passes and targets
$\max_i h_i\le1+\tau_{\rm sol}$.

The direct \texttt{F2PD} implementation instead uses the same nominal time
budget, at most $80$ Newton iterations, residual parameter $10^{-6}$, minimum
barrier parameter $10^{-8}$, and $\tau_0=0.75$. It declares convergence when
the maximum of its primal, dual, and complementarity residuals is below
$10^{-6}(1+\sqrt m)$ and the barrier parameter has reached $10^{-8}$.

These are method-specific stopping targets, not a common output-accuracy
certificate. At an inner time or iteration limit, the code may
continue with a returned primal ellipsoid: PP recovery rescales its covering
ellipsoid to feasibility, while the \texttt{IPOPT} paths also accept a nearly
feasible primal status. Thus a table success means that the outer
relative-volume rule was reached for \texttt{Ellip} or a polarity variant, or
that the stated residual rule was reached for \texttt{F2PD}. In particular,
success for \texttt{Ellip} or a polarity variant does not certify that every
inner solve met its base target.

\paragraph{Reported quantities.}
Tables \ref{tab:benchmark_box_times}--\ref{tab:benchmark_crosspoly_times} report
wall-clock time under the stopping rules above. Each cell records one timed
method call, excluding instance and starting-point construction. Rows use
$(n,m,s)$, with $s\in\{c,m,b\}$ denoting the start. The final rows report
sgm-PAR2: the geometric mean of $p_i+1$, minus one,
using unrounded values, where $p_i$ is the raw time for a successful run and
$2T=3600$ seconds otherwise. Here $<0.01$ denotes a positive time below
$0.01$ seconds; MAX and ERR denote a time limit and an exception.

\begin{table}[!ht]
\centering
\scriptsize
\setlength{\tabcolsep}{4pt}
\renewcommand{\arraystretch}{0.90}
\setlength{\abovecaptionskip}{3pt}
\setlength{\belowcaptionskip}{0pt}
\resizebox{\textwidth}{!}{
\begin{tabular}{crrrrrrr}
\toprule
$(n,m,s)$ & \texttt{Ellip} & \texttt{F2PD} & \texttt{PP-IPOPT} & \texttt{PP-DRN-AS} & \texttt{PP-BCD} & \texttt{PP-WA-TY-Kha} & \texttt{PP-WA-TY-KY} \\
\midrule
$(20,40,c)$ & 0.51 & $<0.01$ & 1.11 & 1.10 & $<0.01$ & $<0.01$ & $<0.01$ \\
$(20,40,m)$ & 0.27 & $<0.01$ & 0.25 & 0.04 & $<0.01$ & $<0.01$ & $<0.01$ \\
$(20,40,b)$ & 0.26 & $<0.01$ & 0.61 & 0.02 & $<0.01$ & $<0.01$ & $<0.01$ \\
\hline
$(40,80,c)$ & 2.63 & $<0.01$ & 0.16 & $<0.01$ & $<0.01$ & $<0.01$ & $<0.01$ \\
$(40,80,m)$ & 4.00 & $<0.01$ & 6.39 & 0.08 & 0.04 & 0.04 & 0.04 \\
$(40,80,b)$ & 4.13 & $<0.01$ & 16.61 & 0.21 & 0.08 & 0.08 & 0.08 \\
\hline
$(60,120,c)$ & 14.00 & $<0.01$ & 0.87 & 0.01 & $<0.01$ & $<0.01$ & $<0.01$ \\
$(60,120,m)$ & 21.36 & $<0.01$ & 63.60 & 0.23 & 0.20 & 0.27 & 0.20 \\
$(60,120,b)$ & 22.26 & $<0.01$ & 151.32 & 0.57 & 0.45 & 0.45 & 0.45 \\
\hline
$(80,160,c)$ & 48.72 & $<0.01$ & 2.60 & $<0.01$ & $<0.01$ & $<0.01$ & $<0.01$ \\
$(80,160,m)$ & 74.20 & $<0.01$ & 309.67 & 0.53 & 0.94 & 0.75 & 0.81 \\
$(80,160,b)$ & 76.43 & $<0.01$ & 779.71 & 1.30 & 1.61 & 1.67 & 1.66 \\
\hline
$(100,200,c)$ & 132.45 & $<0.01$ & 6.25 & 0.03 & 0.01 & 0.01 & 0.01 \\
$(100,200,m)$ & 196.53 & $<0.01$ & 1024.13 & 1.19 & 2.52 & 1.94 & 2.01 \\
$(100,200,b)$ & 202.33 & $<0.01$ & MAX & 2.69 & 4.63 & 4.59 & 4.42 \\
\midrule
sgm-PAR2 & 15.95 & $<0.01$ & 22.04 & 0.41 & 0.42 & 0.40 & 0.40 \\
\bottomrule
\end{tabular}}
\caption{Wall-clock times for the box family.}
\label{tab:benchmark_box_times}
\end{table}

\begin{table}[!ht]
\centering
\scriptsize
\setlength{\tabcolsep}{4pt}
\renewcommand{\arraystretch}{0.90}
\setlength{\abovecaptionskip}{3pt}
\setlength{\belowcaptionskip}{0pt}
\resizebox{\textwidth}{!}{
\begin{tabular}{crrrrrrr}
\toprule
$(n,m,s)$ & \texttt{Ellip} & \texttt{F2PD} & \texttt{PP-IPOPT} & \texttt{PP-DRN-AS} & \texttt{PP-BCD} & \texttt{PP-WA-TY-Kha} & \texttt{PP-WA-TY-KY} \\
\midrule
$(20,40,c)$ & 0.21 & $<0.01$ & 0.01 & $<0.01$ & $<0.01$ & $<0.01$ & $<0.01$ \\
$(20,40,m)$ & 0.23 & $<0.01$ & 0.25 & 0.02 & $<0.01$ & $<0.01$ & $<0.01$ \\
$(20,40,b)$ & 0.28 & $<0.01$ & 0.66 & 0.01 & $<0.01$ & $<0.01$ & $<0.01$ \\
\hline
$(40,80,c)$ & 2.95 & $<0.01$ & 0.20 & $<0.01$ & $<0.01$ & $<0.01$ & $<0.01$ \\
$(40,80,m)$ & 4.40 & $<0.01$ & 6.55 & 0.04 & 0.04 & 0.04 & 0.04 \\
$(40,80,b)$ & 4.58 & $<0.01$ & 15.60 & 0.23 & 0.09 & 0.08 & 0.08 \\
\hline
$(60,120,c)$ & 14.82 & $<0.01$ & 0.91 & $<0.01$ & $<0.01$ & $<0.01$ & $<0.01$ \\
$(60,120,m)$ & 22.28 & $<0.01$ & 69.38 & 0.18 & 0.21 & 0.22 & 0.21 \\
$(60,120,b)$ & 23.29 & $<0.01$ & 134.09 & 0.59 & 0.42 & 0.46 & 0.43 \\
\hline
$(80,160,c)$ & 51.02 & $<0.01$ & 2.91 & $<0.01$ & $<0.01$ & $<0.01$ & $<0.01$ \\
$(80,160,m)$ & 76.77 & $<0.01$ & 254.61 & 0.55 & 0.92 & 0.73 & 0.79 \\
$(80,160,b)$ & 79.46 & $<0.01$ & 712.60 & 1.26 & 1.67 & 1.78 & 1.58 \\
\hline
$(100,200,c)$ & 137.07 & $<0.01$ & 7.41 & 0.03 & 0.01 & 0.01 & 0.01 \\
$(100,200,m)$ & 203.39 & $<0.01$ & 956.01 & 0.78 & 3.79 & 2.05 & 2.51 \\
$(100,200,b)$ & 209.21 & $<0.01$ & MAX & 2.83 & 5.86 & 4.85 & 4.53 \\
\midrule
sgm-PAR2 & 16.36 & $<0.01$ & 20.80 & 0.32 & 0.47 & 0.41 & 0.41 \\
\bottomrule
\end{tabular}}
\caption{Wall-clock times for the skewed-box family.}
\label{tab:benchmark_skewed_box_times}
\end{table}

\begin{table}[!ht]
\centering
\scriptsize
\setlength{\tabcolsep}{4pt}
\renewcommand{\arraystretch}{0.90}
\setlength{\abovecaptionskip}{3pt}
\setlength{\belowcaptionskip}{0pt}
\resizebox{\textwidth}{!}{
\begin{tabular}{crrrrrrr}
\toprule
$(n,m,s)$ & \texttt{Ellip} & \texttt{F2PD} & \texttt{PP-IPOPT} & \texttt{PP-DRN-AS} & \texttt{PP-BCD} & \texttt{PP-WA-TY-Kha} & \texttt{PP-WA-TY-KY} \\
\midrule
$(5,2000,c)$ & 1.78 & 1.79 & 144.49 & 0.04 & 0.95 & 1.89 & 0.91 \\
$(5,2000,m)$ & 1.74 & 2.11 & 268.67 & 0.02 & 1.23 & 2.29 & 1.19 \\
$(5,2000,b)$ & 2.13 & 2.74 & 315.86 & 0.05 & 1.45 & 2.85 & 1.47 \\
\hline
$(10,2000,c)$ & 40.89 & 1.30 & 944.13 & 0.04 & 9.27 & 9.44 & 8.94 \\
$(10,2000,m)$ & 52.91 & 1.50 & 1136.57 & 0.06 & 12.38 & 12.47 & 12.34 \\
$(10,2000,b)$ & 46.03 & 2.06 & 1183.71 & 0.06 & 9.41 & 10.57 & 8.93 \\
\hline
$(20,2000,c)$ & 107.21 & 1.16 & 1373.76 & 0.36 & 8.85 & 11.31 & 8.87 \\
$(20,2000,m)$ & 112.96 & 1.33 & MAX & 0.46 & 11.25 & 14.35 & 10.76 \\
$(20,2000,b)$ & 137.61 & 1.78 & MAX & 0.53 & 12.58 & 17.24 & 12.66 \\
\hline
$(50,5000,c)$ & MAX & 33.99 & MAX & 6.69 & 110.96 & 170.71 & 107.94 \\
$(50,5000,m)$ & MAX & 34.59 & MAX & 10.56 & 151.27 & 238.47 & 168.94 \\
$(50,5000,b)$ & MAX & 18.37 & ERR & 23.49 & 174.77 & 319.59 & 172.25 \\
\midrule
sgm-PAR2 & 86.52 & 3.90 & 1238.73 & 1.13 & 13.33 & 18.32 & 13.23 \\
\bottomrule
\end{tabular}}
\caption{Wall-clock times for the selected random-facet instances.}
\label{tab:benchmark_random_facets_times}
\end{table}

\begin{table}[!htbp]
\centering
\scriptsize
\setlength{\tabcolsep}{4pt}
\resizebox{\textwidth}{!}{
\begin{tabular}{crrrrrrr}
\toprule
$(n,m,s)$ & \texttt{Ellip} & \texttt{F2PD} & \texttt{PP-IPOPT} & \texttt{PP-DRN-AS} & \texttt{PP-BCD} & \texttt{PP-WA-TY-Kha} & \texttt{PP-WA-TY-KY} \\
\midrule
$(5,32,c)$ & 0.01 & $<0.01$ & 1.28 & $<0.01$ & $<0.01$ & $<0.01$ & $<0.01$ \\
$(5,32,m)$ & 0.01 & $<0.01$ & 0.05 & 0.01 & $<0.01$ & $<0.01$ & $<0.01$ \\
$(5,32,b)$ & $<0.01$ & $<0.01$ & 0.07 & $<0.01$ & $<0.01$ & $<0.01$ & $<0.01$ \\
\hline
$(7,128,c)$ & 0.04 & $<0.01$ & 0.02 & 1.10 & $<0.01$ & $<0.01$ & $<0.01$ \\
$(7,128,m)$ & 0.06 & $<0.01$ & 0.34 & 1.17 & $<0.01$ & $<0.01$ & $<0.01$ \\
$(7,128,b)$ & 0.07 & $<0.01$ & 0.94 & 2.81 & $<0.01$ & $<0.01$ & $<0.01$ \\
\hline
$(9,512,c)$ & 0.49 & 0.01 & 0.23 & 7.69 & 0.02 & $<0.01$ & 0.02 \\
$(9,512,m)$ & 0.70 & 0.07 & 51.00 & 2.35 & 0.03 & $<0.01$ & 0.03 \\
$(9,512,b)$ & 0.71 & 0.08 & 63.40 & 2.56 & 0.04 & 0.05 & 0.04 \\
\hline
$(10,1024,c)$ & 1.62 & 0.09 & 0.85 & 32.50 & 0.04 & $<0.01$ & 0.04 \\
$(10,1024,m)$ & 2.47 & 0.14 & 351.18 & 5.55 & 0.06 & 0.03 & 0.06 \\
$(10,1024,b)$ & 2.62 & 0.25 & 259.60 & 9.55 & 0.10 & 0.16 & 0.11 \\
\hline
$(11,2048,c)$ & 7.21 & 0.31 & 3.77 & 41.36 & 0.09 & $<0.01$ & 0.09 \\
$(11,2048,m)$ & 10.93 & 0.87 & 955.77 & 80.58 & 0.14 & 0.09 & 0.13 \\
$(11,2048,b)$ & 11.34 & 1.37 & ERR & 114.83 & 0.23 & 0.58 & 0.26 \\
\hline
$(12,4096,c)$ & 9.79 & 1.60 & 17.92 & 199.66 & 0.22 & $<0.01$ & 0.23 \\
$(12,4096,m)$ & 15.61 & 5.13 & 1737.79 & 130.99 & 0.36 & 0.33 & 0.37 \\
$(12,4096,b)$ & 16.84 & 8.28 & MAX & 144.97 & 0.86 & 2.31 & 0.84 \\
\midrule
sgm-PAR2 & 2.10 & 0.50 & 23.39 & 10.10 & 0.11 & 0.14 & 0.11 \\
\bottomrule
\end{tabular}}
\caption{Wall-clock times for the cross-polytope family.}
\label{tab:benchmark_crosspoly_times}
\end{table}

\paragraph{Observed trends.}
The polarity process is a versatile framework rather than one fixed algorithm.
Changing the inner MinCE oracle lets the
same outer transformation use a generic interior-point solve, a reduced
Newton active-set method, coordinate steps, or away-step first-order methods.
Under the implementation and stopping conventions above, \texttt{F2PD} records
especially short times on the structured box and skewed-box cases, the dual
reduced Newton active-set oracle on the dense random-facet cases, and the WA-TY
away-step and BCD oracles on the cross-polytope cases. No method is fastest
on every tested family. These deterministic single-seed cases
illustrate how performance can depend on matching the MinCE solver to the
instance geometry.

\section{Conclusion} \label{sec:conclusion}
This paper studies the polarity process for computing a polytope's MaxIE. The
main contribution is a log-volume-potential analysis, independent of prior
asymptotic-convergence
arguments, that establishes global linear contraction of the potential gap along each trajectory.
Its existential, instance-dependent rate is determined by the initial sublevel
geometry and yields a finite volume-ratio bound.

We also analyze an inexact polarity process with approximate polar MinCE
subproblems. The modular framework accepts any MinCE oracle meeting the
required outer lifted-objective tolerance. Combined with existing MinCE
algorithms, it gives conditional arithmetic estimates. The implementation cases
show that recorded performance depends on matching the MinCE oracle to the
instance geometry.

\enlargethispage{3\baselineskip}
\begin{acknowledgements}
The author thanks Prof. X. Andy Sun for introducing the MaxIE problem, sharing his unpublished notes,
engaging in many insightful discussions, and encouraging the author to publish this work. 

During the preparation of this manuscript, the author used ChatGPT and Codex (OpenAI) for 
language editing, reference checking, coding assistant, and as an auxiliary tool for reviewing the exposition and checking mathematical arguments. All mathematical results and proofs were developed and independently verified by the author, who takes full responsibility for the accuracy and integrity of the manuscript.
\end{acknowledgements}

\begingroup
\linespread{0.986}\selectfont
\bibliographystyle{spmpsci}
\bibliography{ref.bib}

\begin{thebibliography}{10}
\providecommand{\url}[1]{{#1}}
\providecommand{\urlprefix}{URL }
\expandafter\ifx\csname urlstyle\endcsname\relax
  \providecommand{\doi}[1]{DOI~\discretionary{}{}{}#1}\else
  \providecommand{\doi}{DOI~\discretionary{}{}{}\begingroup \urlstyle{rm}\Url}\fi

\bibitem{ahipasaoglu2008fw}
Ahipasaoglu, S.D., Sun, P., Todd, M.J.: Linear convergence of a modified {Frank--Wolfe} algorithm for computing minimum-volume enclosing ellipsoids.
\newblock Optimization Methods and Software \textbf{23}(1), 5--19 (2008).
\newblock \doi{10.1080/10556780701589669}

\bibitem{anstreicher2002improved}
Anstreicher, K.M.: Improved complexity for maximum-volume inscribed ellipsoids.
\newblock SIAM Journal on Optimization \textbf{13}(2), 309--320 (2002)

\bibitem{aubrun2017alice}
Aubrun, G., Szarek, S.J.: Alice and Bob Meet Banach: The Interface of Asymptotic Geometric Analysis and Quantum Information Theory, \emph{Mathematical Surveys and Monographs}, vol. 223.
\newblock American Mathematical Society (2017).
\newblock \doi{10.1090/surv/223}

\bibitem{bertsekas1999nonlinear}
Bertsekas, D.P.: Nonlinear Programming, 2nd edn.
\newblock Athena Scientific (1999)

\bibitem{bessaga1959converse}
Bessaga, C.: On the converse of the {Banach} ``fixed-point principle''.
\newblock Colloquium Mathematicum \textbf{7}, 41--43 (1959).
\newblock \doi{10.4064/cm-7-1-41-43}

\bibitem{bezanson2017julia}
Bezanson, J., Edelman, A., Karpinski, S., Shah, V.B.: Julia: A fresh approach to numerical computing.
\newblock SIAM Review \textbf{59}(1), 65--98 (2017)

\bibitem{CaoLiSongYangZhou2025}
Cao, Y., Li, X., Song, Z., Yang, X., Zhou, T.: Faster algorithms for structured {John} ellipsoid computation.
\newblock In: Advances in Neural Information Processing Systems, vol.~38 (2025)

\bibitem{CohenCousinsLeeYang2019}
Cohen, M.B., Cousins, B., Lee, Y.T., Yang, X.: A near-optimal algorithm for approximating the {John} ellipsoid.
\newblock In: Proceedings of the Thirty-Second Conference on Learning Theory, \emph{Proceedings of Machine Learning Research}, vol.~99, pp. 849--873 (2019)

\bibitem{cousins2022volumeandsampling}
Cousins, B.: {Volume-and-Sampling}.
\newblock GitHub repository (2022).
\newblock Version 2.2.1, \url{https://github.com/Bounciness/Volume-and-Sampling}

\bibitem{grotschel1988geometric}
Gr{\"o}tschel, M., Lov{\'a}sz, L., Schrijver, A.: Geometric Algorithms and Combinatorial Optimization, \emph{Algorithms and Combinatorics}, vol.~2.
\newblock Springer (1988)

\bibitem{horn2012matrix}
Horn, R.A., Johnson, C.R.: Matrix Analysis, 2nd edn.
\newblock Cambridge University Press (2012)

\bibitem{john1948extremum}
John, F.: Extremum problems with inequalities as subsidiary conditions.
\newblock In: Studies and Essays Presented to R. Courant on His 60th Birthday, January 8, 1948, pp. 187--204. Interscience Publishers (1948)

\bibitem{khachiyan1996rounding}
Khachiyan, L.G.: Rounding of polytopes in the real number model of computation.
\newblock Mathematics of Operations Research \textbf{21}(2), 307--320 (1996)

\bibitem{khachiyan1990complexity}
Khachiyan, L.G., Todd, M.J.: On the complexity of approximating the maximal inscribed ellipsoid for a polytope.
\newblock Mathematical Programming \textbf{61}(1-3), 137--159 (1993)

\bibitem{kumar2005minimum}
Kumar, P., Y{\i}ld{\i}r{\i}m, E.A.: Minimum-volume enclosing ellipsoids and core sets.
\newblock Journal of Optimization Theory and Applications \textbf{126}(1), 1--21 (2005).
\newblock \doi{10.1007/s10957-005-2653-6}

\bibitem{kurzhanski1997ellipsoidal}
Kurzhanski, A.B., V{\'a}lyi, I.: Ellipsoidal Calculus for Estimation and Control.
\newblock Systems \& Control: Foundations \& Applications. Birkh{\"a}user, Boston, MA (1997).
\newblock \doi{10.1007/978-1-4612-0277-6}

\bibitem{LiYuJiangGaoHan2026}
Li, X., Yu, J., Jiang, J., Gao, J., Han, A.: Beyond averaging in {John} ellipsoid approximation: High-accuracy algorithms in the leverage-score model.
\newblock arXiv:2606.20082  (2026)

\bibitem{nesterov1994interior}
Nesterov, Y., Nemirovskii, A.: Interior-point polynomial algorithms in convex programming.
\newblock SIAM (1994)

\bibitem{nesterov1989self}
Nesterov, Y.E., Nemirovskii, A.S.: Self-concordant functions and polynomial-time methods in convex programming.
\newblock Tech. rep., Central Economic and Mathematical Institute, USSR Academy of Sciences, Moscow, USSR (1989)

\bibitem{sunfreund2004mvce}
Sun, P., Freund, R.M.: Computation of minimum-volume covering ellipsoids.
\newblock Operations Research \textbf{52}(5), 690--706 (2004)

\bibitem{wachter2006implementation}
W{\"a}chter, A., Biegler, L.T.: On the implementation of an interior-point filter line-search algorithm for large-scale nonlinear programming.
\newblock Mathematical Programming \textbf{106}(1), 25--57 (2006)

\bibitem{WoodruffYasuda2024}
Woodruff, D.P., Yasuda, T.: {John} ellipsoids via lazy updates.
\newblock In: Advances in Neural Information Processing Systems, vol.~37 (2024).
\newblock \doi{10.52202/079017-2179}

\bibitem{xie2006maxsphere}
Xie, Y., Snoeyink, J., Xu, J.: Efficient algorithm for approximating maximum inscribed sphere in high dimensional polytope.
\newblock In: Proceedings of the Twenty-Second Annual Symposium on Computational Geometry, pp. 21--29 (2006)

\bibitem{zhanggao2003mvie}
Zhang, Y., Gao, L.: On the numerical solution of the maximum-volume ellipsoid problem.
\newblock SIAM Journal on Optimization \textbf{14}(1), 53--76 (2003)

\bibitem{Zhao2025AwayStep}
Zhao, R.: New analysis of an away-step {Frank--Wolfe} method for minimizing log-homogeneous barriers.
\newblock Mathematics of Operations Research \textbf{51}(1), 35--59 (2025).
\newblock \doi{10.1287/moor.2023.0281}

\end{thebibliography}
\endgroup

\end{document}